\documentclass[11pt]{amsart}

\usepackage[margin = 1in]{geometry}
\usepackage{amsfonts, amsmath,amscd, amssymb, cite, color, latexsym, mathrsfs, mathtools,
slashed, stmaryrd, verbatim, wasysym }
\usepackage[all]{xy}
\usepackage[pdftex]{graphicx}
\usepackage{hyperref}
\usepackage[applemac]{inputenc}

\newtheorem{theorem}{Theorem}
\newtheorem{corollary}[theorem]{Corollary}
\newtheorem{definition}{Definition}

\newtheorem{lemma}[theorem]{Lemma}
\newtheorem{proposition}[theorem]{Proposition}

\numberwithin{equation}{section}
\numberwithin{theorem}{section}

\let\oldsqrt\sqrt
\def\sqrt{\mathpalette\DHLhksqrt}
\def\DHLhksqrt#1#2{%
\setbox0=\hbox{$#1\oldsqrt{#2\,}$}\dimen0=\ht0
\advance\dimen0-0.2\ht0
\setbox2=\hbox{\vrule height\ht0 depth -\dimen0}%
{\box0\lower0.4pt\box2}}

\renewcommand{\tilde}{\widetilde}

\newcommand{\wt}[1]{\widetilde{#1}}

\newcommand\pa{\partial}

\newcommand\eps\varepsilon
\renewcommand\epsilon\varepsilon

\newcommand\Ch{\operatorname{Ch}}

\newcommand\dvol{\operatorname{dvol}}

\newcommand\Id{\operatorname{Id}}
\renewcommand\Im{\operatorname{Im}}

\newcommand\Ind{\operatorname{Ind}}

\newcommand\Ker{\operatorname{Ker}}

\newcommand\reg{\operatorname{reg}}
\newcommand\sing{\operatorname{sing}}

\newcommand\paperintro%
        {%
         }
\newcommand\paperbody%
        {%
         }

\newcommand\bbC{\mathbb{C}}

\DeclareMathAlphabet{\mathpzc}{OT1}{pzc}{m}{it}

\DeclareMathOperator{\im}{im}      

\begin{document}
\pagestyle{myheadings}


\title{On analytic Todd classes of singular varieties}

\author{Francesco Bei and Paolo Piazza}
\address{Dipartimento di Matematica, Sapienza Universit\`a di Roma}
\email{bei@mat.uniroma1.it}
\address{Dipartimento di Matematica, Sapienza Universit\`a di Roma}
\email{piazza@mat.uniroma1.it}
\maketitle

\begin{abstract}
Let $(X,h)$ be a compact and irreducible Hermitian complex space. This paper is devoted to various questions concerning the analytic K-homology of $(X,h)$. In the fist part, assuming either $\dim(\sing(X))=0$ or $\dim(X)=2$, we show that the rolled-up operator of the minimal $L^2$-$\overline{\partial}$ complex, denoted here $\overline{\eth}_{\mathrm{rel}}$, induces a class in $K_0 (X)\equiv KK_0(C(X),\mathbb{C})$. A similar result, assuming $\dim(\sing(X))=0$,  is proved also for $\overline{\eth}_{\mathrm{abs}}$, the rolled-up operator of the maximal $L^2$-$\overline{\partial}$ complex. We then show that when $\dim(\sing(X))=0$  we have $[\overline{\eth}_{\mathrm{rel}}]=\pi_*[\overline{\eth}_M]$ with
 $\pi:M\rightarrow X$ an arbitrary resolution and with $[\overline{\eth}_M]\in K_0 (M)$ 
the analytic K-homology class induced by $\overline{\partial}+\overline{\partial}^t$ on $M$. In the second part of the paper
we focus on complex projective varieties $(V,h)$ endowed with the Fubini-Study metric. First, assuming $\dim(V)\leq 2$, we compare the Baum-Fulton-MacPherson K-homology class of $V$ with the class defined analytically through the rolled-up operator of any $L^2$-$\overline{\partial}$ complex. We show that there is no $L^2$-$\overline{\partial}$ complex on $(\reg(V),h)$ whose rolled-up operator induces a K-homology class that equals the Baum-Fulton-MacPherson class. Finally in the last part of the paper we prove that under suitable assumptions on $V$ the push-forward of  $[\overline{\eth}_{\mathrm{rel}}]$ 
in the K-homology of the classifying space of the fundamental group of $V$ is a  birational invariant.
\end{abstract}

\vspace{0.7 cm}

\noindent\textbf{Keywords}: Complex spaces, projective varieties, resolution of singularities, analytic K-homology, Baum-Fulton-MacPherson class, birational invariants.
\vspace{0.4 cm}

\noindent\textbf{Mathematics subject classification (2010)}:   32W05, 32C15, 32C18, 19L10, 14C40.

\tableofcontents

\section{Introduction}
Let $X$ be a smooth compact riemannian manifold without boundary and let $\eth$ be a Dirac-type operator acting on the sections
of a bundle of Clifford modules $E\to X$. We assume that $\eth$ is formally self-adjoint: $\eth=\eth^t$.
We let 
$\Gamma:= \pi_1 (X)$ and consider the associated universal covering $\widetilde{X}$. Finally, we consider the classifying space $B\Gamma$ 
and denote by  $r:X\to B\Gamma$ a classifying map for $\widetilde{X}\to X$. It is well known that 
$\eth: C^\infty (X,E)\subset L^2 (X,E)\to L^2 (X,E)$ is an essentially self-adjoint operator. Its unique self-adjoint extension,
still denoted $\eth$,
is Fredholm on its domain endowed with the graph norm. More generally, to $\eth$ we can associate an unbounded Kasparov $C(X)$-module and thus a class $[\eth]\in K_* (X):=KK_* (C(X),\mathbb{C})$, $*=\dim_{\mathbb{R}} X\;{\rm mod}\;2$. Notice that the existence of this  K-theory class  involves analytic properties 
of $\eth$ that are finer than the Fredholm property alone (for example, the compactness of the resolvant associated to $\eth$).\\
Three examples must be singled out because  of their deep connections with geometric and topological 
properties of $X$:

\begin{itemize}
\item the signature operator $\eth^{{\rm sign}}:\Omega^{\bullet} (X) \to \Omega^{\bullet} (X) $, when $X$ is oriented
\item the spin-Dirac operator $\eth^{{\rm spin}}: C^\infty (X,S)\to C^\infty(X,S)$, when $X$ is spin 
\item the operator $\overline{\partial}+\overline{\partial}^t: \Omega^{0,\bullet}(X) \to \Omega^{0,\bullet} (X)$, when $X\subset \mathbb{C}\mathbb{P}^n$ is a smooth projective variety
endowed with the restriction of a Hermitian metric on $\mathbb{C}\mathbb{P}^n$.
\end{itemize}
 Notice that because of the homotopy invariance of K-homology the classes 
 $[\eth^{{\rm sign}}]$, $[\eth^{{\rm spin}}]$, $[\overline{\partial}+\overline{\partial}^t]$ 
 associated to these three 
 operators in $K_* (X)$ are independent of the metric that we have used in order to define them.
 
 \medskip
 \noindent
 We call  $[\overline{\partial}+\overline{\partial}^t]\in K_0 (X)$  the {\it analytic Todd class} of the projective variety $X$.
 
 \medskip
 \noindent
 The Atiyah-Singer index theorem for the twisted versions of these operators can be used in order to
 establish the following  fundamental equalities in $H_* (X,\mathbb{Q})$:
 \begin{equation}\label{chern-ch}
 \Ch_* [\eth^{{\rm sign}} ]=L_* (X)\qquad \Ch_* [\eth^{{\rm spin}} ]=\widehat{A}_* (X)\qquad \Ch_* [ \overline{\partial}+\overline{\partial}^t]={\rm Td}_* (X)\,,
 \end{equation}
 with $L_* (X)$, $\widehat{A}_* (X)$ and ${\rm Td}_* (X)$ the homology classes obtained as the Poincar\'e duals
 of their well-known cohomology counterparts $L (X)$, $\widehat{A} (X)$ and ${\rm Td} (X)$ in $H^* (X,\mathbb{Q})$.\\
These three equalities can be complemented by three stability properties for the homology classes $r_* [\eth^{{\rm sign}} ]$,
$r_* [\eth^{{\rm spin}} ]$ and  $ r_* [ \overline{\partial}+\overline{\partial}^t]$ obtained by pushing forward to $K_* (B\Gamma)$:\\
assume that the fundamental group
$\Gamma$ satisfies the Strong Novikov Conjecture \footnote{this means that the assembly map $K_* (B\Gamma)\to K_* (C^*_r\Gamma)$ is rationally injective; the Strong Novikov Conjecture is satisfied by large classes of groups and 
no counterexamples to it are known}, then \cite{Kasparov-inventiones}, \cite{RosenbergPSC}, \cite{RosenbergPSCIII}
\cite{rosenberg-tams}
\begin{itemize}
\item $r_* [\eth^{{\rm sign}} ]\in H_* (B\Gamma,\mathbb{Q})$ is an oriented homotopy invariant of $X$
\item
 $r_* [\eth^{{\rm spin}} ]\in H_* (B\Gamma, \mathbb{Q})$ vanishes if the riemannian metric defining $\eth^{{\rm spin}}$
is of positive scalar curvature
\item  $r_* [ \overline{\partial}+\overline{\partial}^t]\in H_* (B\Gamma, \mathbb{Q})$ is a birational invariant of the smooth 
projective variety $X$.
\end{itemize}
In fact, for the third example, the one stating the birational invariance of  $r_* [ \overline{\partial}+\overline{\partial}^t]$
in $H_* (B\Gamma,\mathbb{Q})$, we do not need any assumption on $\Gamma$, see \cite{WBlock}, \cite{BraSchYou} and  \cite{Hilsumproj}.\\
The equalities \eqref{chern-ch} together with these stability results then imply the following fascinating statements:

\begin{itemize}
\item the numbers $\{\langle \alpha, r_* (L_* (X))\rangle,\,\alpha\in H^* (B\Gamma,\mathbb{Q})\}$ are oriented homotopy invariants
\item the numbers $\{\langle \alpha, r_* (\widehat{A}_* (X))\rangle,\,\alpha\in H^* (B\Gamma,\mathbb{Q})\}$ are topological
obstructions to the existence of a metric of positive scalar curvature on the spin manifold $X$
\item  the numbers $\{\langle \alpha, r_* ({\rm Td}_* (X))\rangle,\,\alpha\in H^* (B\Gamma,\mathbb{Q})\}$ are birational invariants
of the smooth projective variety $X$.
\end{itemize}
These numbers constitute  respectively the {\it higher signatures}, the {\it higher $\widehat{A}$-genera} and the {\it higher 
Todd genera} of $X$; using Poincar\'e duality they can also be expressed as
$$\int_X L(X)\wedge r^* \alpha\,,\qquad \int_X \widehat{A}(X)\wedge r^* \alpha\,,\qquad \int_X {\rm Td}(X)\wedge r^* \alpha
.$$

In the past  fourty  years 
 a great effort has been devoted to the study of the analytic, geometric and 
topological properties of Dirac operators on the regular part of a stratified pseudomanifold. References are too numerous 
to be recorder here.
One might therefore wonder which of the above properties can be extended to the stratified category.

For the signature operator associated to a wedge metric \footnote{also called an {\it iterated incomplete edge metric} or
an {\it iterated conic metric}}
these particular questions were tackled in \cite{ALMP11} under the assumption that $X$ is either a Witt space or, more generally, 
a Cheeger space \cite{ALMP13.1}: in these two cases topological L-classes had been previously defined by Goresky-MacPherson on Witt spaces  and by Banagl
on Cheeger spaces \footnote{Cheeger also proposed a definition of the homology
L-class of a Witt space, see \cite{Cheeger:Spec}}, we denote them $L_*^{{\rm GM}}$ and $L_*^{{\rm B}}$,  and one of the main results of \cite{ALMP11} 
 \cite{ALMP13.2} \cite{ALMP13.1} was the definition of a K-homology analytic signature class satisfying the analogue of \eqref{chern-ch} and with  stability properties similar to the one stated above for its
push-forward to $K_* (B\Gamma)$.
This established, in particular, the following result:\\ the  higher signatures of a Witt space
or of a Cheeger space,
$$\{\langle \alpha, r_* (L_*^{{\rm GM}} (X))\rangle,\,\alpha\in H^* (B\Gamma,\mathbb{Q})\}\qquad 
 \{\langle \alpha, r_* (L_*^{{\rm B}} (X))\rangle,\,\alpha\in H^* (B\Gamma,\mathbb{Q})\},$$
are {\it stratified} homotopy invariants.

 For the other two examples the situation is less satisfactory. For a spin stratified pseudomanifold, i.e. a pseudomanifold
 with all strata spin, there are interesting recent results by Albin and Gell-Redmann \cite{AlbinGell}: if the wedge metric induced along the
 links is of positive scalar curvature, then there is a well defined K-homology class and it is true that its push-forward
 in $K_* (B\Gamma)$ is an obstruction to the existence of a wedge metric of positive scalar curvature if $\Gamma$ satisfies the Strong Novikov Conjecture \footnote{these
 results are not explicitly stated in \cite{AlbinGell} but they follow from the analysis developed  there and from the arguments
 given in \cite{ALMP11}  for the signature operator}. A different approach to these results, using groupoids and iterated 
 $\Phi$-metrics, can be found in \cite{PiZe}. However,
 what is missing in the spin case is a {\it topological} definition of the homology $\widehat{A}$-class of a stratified 
 pseudomanifold.
 
 We finally come to the last
 example: a {\it singular} projective variety $X\subset \mathbb{C}\mathbb{P}^n$ endowed in its regular part with the Hermitian metric induced by a Hermitian
 metric $h$ in $\mathbb{C}\mathbb{P}^n$, for example the Fubini-Study metric, and the associated operator $ \overline{\partial}+\overline{\partial}^t$ on it. The analysis for this operator is notoriously more difficult than in the two preceeding
 examples: this is due to the non-product nature of the metric near the singular locus and to the fact that already in simple examples, such as singular algebraic curves, the operator $ \overline{\partial}+\overline{\partial}^t$ fails to be
 essentially self-adjoint. Still, many interesting papers have been 
 devoted to the analysis of $ \overline{\partial}+\overline{\partial}^t$ on singular projective varieties, albeit never
 in the generality one would like to consider. See for example \cite{Bei2017}, \cite{FraBei}, \cite{BPScurves}, \cite{FoxHaskellHodge}, \cite{GrieserLesch}, \cite{LiTian}, \cite{Nagase},  \cite{OvRu}, \cite{OvreVaAJM}, \cite{OvreVaIn}, \cite{PSCrelle}, \cite{Pardon},  \cite{PardonSternJAMS},   \cite{RUJFA}, \cite{JRupp}, \cite{RUIMRN} and many others. Among  the papers devoted  to  $ \overline{\partial}+\overline{\partial}^t$ on singular projective varieties few are centered around  the problem of defining  K-homology
 classes in $K_* (X)$ 
 and studying their properties (the plural is employed here  because, as we have already pointed out, there are a priori
 different self-adjoint extensions of $ \overline{\partial}+\overline{\partial}^t$). We refer the reader to the work of 
  Haskell \cite{Haskell-K},\cite{HaskellIndex} 
   and Fox-Haskell  \cite{FoxHaskellTodd}.
It is important to notice that 
 for a singular projective variety $X$ we {\it do have} a topologically defined
 homology Todd class: this is the Baum-Fulton-MacPherson class, denoted in this paper  by 
 ${\rm Td}_*^{{\rm BFM}} (X)\in H_* (X,\mathbb{Q})$, see \cite{BFMI}. This homology Todd class is in fact equal to
 the Chern character of a topological K-homology Todd class, denoted here 
 ${\rm Td}_K^{{\rm BFM}}(X)\in K_* (X)$. See \cite{BFMII}.
 We can finally state the purposes of the present work:
 
 \bigskip
 \noindent
 {\it the main goal of this article is to define analytic K-homology classes associated to  
 $ \overline{\partial}+\overline{\partial}^t$ 
 on a singular projective variety and to study their properties; in particular
 \begin{itemize}
\item   their relationship with the K-homology class of a Hironaka resolution; 
\item the birational invariance of their push-forward to $K_* (B\Gamma)$;
\item the relationship
 of their Chern character with the Baum-Fulton-MacPherson homology Todd class ${\rm Td}_*^{{\rm BFM}} (X)\in H_* (X,\mathbb{Q})$.
 \end{itemize}}
 
 \bigskip
 \noindent
 We shall now very briefly illustrate the main results of this work. We consider a singular projective variety
 endowed with the Hermitian metric induced by a Hermitian metric on $\mathbb{C}\mathbb{P}^n$  \footnote{in fact, in this first part of the article, we work more generally with complex hermitian spaces.}. We begin by defining two self-adjoint extensions of
 $ \overline{\partial}+\overline{\partial}^t$; these are obtained by considering respectively the rolled-up operator of the minimal and the maximal extension
 of the complex $\overline{\partial}_{0,q}$. We prove that if the singular locus $\sing (X)$ is zero dimensional then
 these extensions  define K-homology classes in $K_* (X)$ and we denote these classes respectively by
 $[\overline{\eth}_{{\rm rel}}]$ and $[\overline{\eth}_{{\rm abs}}]$. For complex surfaces we show the existence of  $[\overline{\eth}_{{\rm rel}}]$ 
 without any assumption on the singular locus. 
 All these results employ previous work of Bei and of {\O}vrelid-Ruppenthal \cite{Bei2017}, \cite{FraBei}, \cite{OvRu}. 
 Next, requiring again $\dim(\sing(X))=0$, we show that $[\overline{\eth}_{{\rm rel}}]$ is equal to $\pi_* [\overline{\partial}_M + 
 \overline{\partial}^t_M]$ with $\pi: M\to X$ any Hironaka resolution of $X$. In this part of the article we make use of results
of Peter Haskell \cite{Haskell-K} \footnote{For this part of the article we point out that similar (in fact stronger) results have been announced
by Hilsum in a seminar in Shanghai in 2017}. Using previous work of ours \cite{BeiPiazza} and of  Timmerscheidt \cite{Esnault}, we also give other descriptions
 of $\pi_* [\overline{\partial}_M + 
 \overline{\partial}^t_M]$, using Saper-type metrics and Poincar\'e-type metrics on $\reg(X)$.
 \\ We finally come to
 the relationship of our K-homology classes with the one defined by Baum-Fulton-MacPherson, ${\rm Td}_K^{{\rm BFM}}(X)\in K_* (X)$. We show that if $X$ is a singular  algebraic curve and if $\overline{D}$ is any closed
extension  of $\overline{\partial}:C^{\infty}_c(\reg(X))\rightarrow \Omega^{0,1}_c(\reg(X))$ then  
\begin{equation}\label{intro:not=}\Ind(\overline{D})\not=\chi(X,\mathcal{O}_X).
\end{equation} We prove that such extension $\overline{D}$
always defines a K-homology class $[\overline{D}+\overline{D}^*]$  in $K_* (X)$;
since $ \chi(X,\mathcal{O}_X)=p_*  ( {\rm Td}_K^{{\rm BFM}}(X) )$ and since
 $\Ind(\overline{D})=p_*[\overline{D}+\overline{D}^*]$
 with $p: X\to {\rm point}$, we see from \eqref{intro:not=} that 
  \begin{equation}\label{intro:not=K}
  [\overline{D}+\overline{D}^*]\not=  {\rm Td}_K^{{\rm BFM}}(X)\quad\text{in}\quad K_* (X).
  \end{equation}
Similarly, let $X$ be a normal complex surface with isolated singularities and with Hironaka resolution
$M$; assume that   $R^1 \pi_* \mathcal{O}_M$ is non-trivial. Then   we prove that
for any closed extension $(L^2\Omega^{0,q}(\reg(X),h),\overline{D}_{0,q})$ of the complex $(\Omega_c^{0,q}(\reg(X)),\overline{\partial}_{0,q})$ we have 
 \begin{equation}\label{intro:not=surface}
\chi_{2,\overline{D}_{0,q}}(\reg(X),h)\not= \chi(X,\mathcal{O}_X).
\end{equation}
Also in this case we show that the rolled-up operator associated to $\overline{D}_{0,q}$ defines
a K-homology class $[\overline{D}_0 + \overline{D}_0^*]\in K_* (X)$.
Using  \eqref{intro:not=surface} we see that  if $X$ is normal and $R^1\pi_*\mathcal{O}_M$ is non-trivial then
for any closed extension 
$\overline{D}_{0,q}$ 
of $(\Omega_c^{0,q}(\reg(X)),\overline{\partial}_{0,q})$ with associated K-homology class  $[\overline{D}_0 + \overline{D}_0^*]\in K_* (X)$ we have that 
\begin{equation}\label{intro:not=Ksurfaces}
  [\overline{D}_0+\overline{D}_0^*]\not=  {\rm Td}_K^{{\rm BFM}}(X)\quad\text{in}\quad K_* (X).
  \end{equation}
  One might then ask  whether there is a Hilbert complex, from the above results necessarily different from 
   $(L^2\Omega^{0,q}(\reg(X),h),\overline{D}_{0,q})$, with the property that its rolled up operator defines a K-homology class
  which realizes analytically the Baum-Fulton-MacPherson class. John Lott has recently constructed 
  such a complex.  See \cite{Lott}.

%
%
In the last part of the paper we specialize to singular projective varieties with only isolated {\em rational}
 singularities and we show that in this case 
 $$[\overline{\eth}_{{\rm rel}}]=  {\rm Td}_K^{{\rm BFM}}(X)\quad\text{in}\quad K_* (X).$$
Finally  the last section is devoted to singular projective varieties with only isolated singularities admitting a resolution that induces an isomorphism of fundamental groups. Interesting examples are provided by complex projective surfaces with rational singularities and  projective varieties with  log-terminal singularities, see \cite{Briesk} and \cite{Takayama}, respectively. We will show that in this setting the class  $r_* [\overline{\eth}_{{\rm rel}}]\in K_* (B\Gamma)$ is a birational invariant.

 
 \medskip
 \noindent
 {\bf Acknowledgements.} We are happy to thank Peter Haskell for interesting email correspondence. 
 We thank Marco Manetti, Kieran O'Grady, Luca Migliorini and Shoji Yokura for useful information related to the
 Baum-Fulton-MacPherson class. The first author wishes to thank Stefano Urbinati and Ernesto Mistretta for helpful discussions.   Part of this work was done during visits of the first author to 
 Sapienza Universit\`a di Roma; the  financial support of Sapienza is  gratefully  acknowledged. The first author wishes also to thank the university of Padova for financial support. Finally this work was partially performed within the framework of the LABEX MILYON (ANR-10-LABX-0070) of Universit\'e de Lyon, within the program ''Investissements d'Avenir'' (ANR-11-IDEX-0007) operated by the French National Research Agency (ANR).

\section{Background material}
This section provides a very concise summary of  the basic properties of the $L^2$-$\overline{\partial}$-cohomology and
of  the $L^2$-closures of the operators $\overline{\pa}_{p,q}$, $\overline{\pa}_p+\overline{\pa}_p^t$ and $\Delta_{\overline{\partial},p,q}$ over a possible incomplete Hermitian manifold. The proofs of the statements we are going to recall can be found in \cite{BruLe}. Let $(M,g)$ be a complex manifold of real dimension $2m$.  With $L^2\Omega^{p,q}(M,g)$ we denote the Hilbert space of $L^2$-$(p,q)$-forms. The Dolbeault operator acting on $(p,q)$-forms is labeled by $\overline{\partial}_{p,q}:\Omega^{p,q}(M)\rightarrow \Omega^{p,q+1}(M)$. When we look at $\overline{\partial}_{p,q}:L^2\Omega^{p,q}(M,g)\rightarrow L^2\Omega^{p,q+1}(M,g)$ as an unbounded and densely defined operator with domain $\Omega_c^{p,q}(M)$ we label by $\overline{\partial}_{p,q,\max/\min}:L^2\Omega^{p,q}(M,g)\rightarrow L^2\Omega^{p,q+1}(M,g)$ respectively its maximal and minimal extension. With $\overline{\partial}_{p,q}^t:\Omega^{p,q+1}_c(M)\rightarrow \Omega^{p,q}_c(M)$ we denote the formal adjoint of $\overline{\partial}_{p,q}$. For each bidegree $(p,q)$ we have  the Hodge-Kodaira Laplacian  defined as $$\Delta_{\overline{\partial},p,q}:\Omega^{p,q}_c(M)\rightarrow \Omega^{p,q}_c(M),\ \Delta_{\overline{\partial},p,q}:=\overline{\partial}_{p,q-1}\circ\overline{\partial}^t_{p,q-1}+\overline{\partial}_{p,q}^t\circ \overline{\partial}_{p,q}.$$
In the case of functions, that is $(p,q)=(0,0)$, we will simply write $\Delta_{\overline{\partial}}:C^{\infty}_c(M)\rightarrow C^{\infty}_c(M)$.
We recall now the definition of the  following two self-adjoint extensions of $\Delta_{\overline{\partial},p,q}$:
\begin{equation}
\label{asdf}
\overline{\partial}_{p,q-1,\max}\circ \overline{\partial}_{p,q-1,\min}^t+\overline{\partial}_{p,q,\min}^t\circ \overline{\partial}_{p,q,\max}:L^2\Omega^{p,q}(M,g)\rightarrow L^2\Omega^{p,q}(M,g)
\end{equation} 
and 
\begin{equation}
\label{basdf}
\overline{\partial}_{p,q-1,\min}\circ \overline{\partial}_{p,q-1,\max}^t+\overline{\partial}_{p,q,\max}^t\circ \overline{\partial}_{p,q,\min}:L^2\Omega^{p,q}(M,g)\rightarrow L^2\Omega^{p,q}(M,g)
\end{equation}
called respectively the absolute and the relative extension. The operator \eqref{asdf}, the absolute extension, is labeled in general with  $\Delta_{\overline{\partial},p,q,\text{abs}}$ and its domain is given by $\mathcal{D}(\Delta_{\overline{\partial},p,q,\text{abs}})=$ $\{\omega\in \mathcal{D}(\overline{\partial}_{p,q,\max})\cap \mathcal{D}(\overline{\partial}_{p,q-1,\min}^t):\overline{\partial}_{p,q,\max}\omega \in \mathcal{D}(\overline{\partial}^t_{p,q,\min})\ \text{and}\ \overline{\partial}_{p,q-1,\min}^t\omega \in \mathcal{D}(\overline{\partial}_{p,q-1,\max})\}.$
The operator \eqref{basdf}, the relative extension,  is labeled in general with  $\Delta_{\overline{\partial},p,q,\text{rel}}$ and its domain is given by $\mathcal{D}(\Delta_{\overline{\partial},p,q,\text{rel}})=\{\omega\in \mathcal{D}(\overline{\partial}_{p,q,\min})\cap \mathcal{D}(\overline{\partial}_{p,q-1,\max}^t):\overline{\partial}_{p,q,\min}\omega \in \mathcal{D}(\overline{\partial}^t_{p,q,\max})\ \text{and}\ \overline{\partial}_{p,q-1,\max}^t\omega \in \mathcal{D}(\overline{\partial}_{p,q-1,\min})\}.$
 The kernels of  $\Delta_{\overline{\partial},p,q,\text{abs}}$ and $\Delta_{\overline{\partial},p,q,\text{rel}}$ are denoted with
$\mathcal{H}^{p,q}_{\overline{\partial},\text{abs}}(M,g)$ and $\mathcal{H}^{p,q}_{\overline{\partial},\text{rel}}(M,g)$ respectively and they satisfies $\mathcal{H}^{p,q}_{\overline{\partial},\text{abs}}(M,g)=\ker(\overline{\partial}_{p,q,\max})\cap \ker(\overline{\partial}^t_{p,q-1,\min})$  and  $\mathcal{H}^{p,q}_{\overline{\partial},\text{rel}}(M,g)=\ker(\overline{\partial}_{p,q,\min})\cap \ker(\overline{\partial}^t_{p,q-1,\max})$. Consider now the Hodge-Dolbeault operator $\overline{\partial}_{p}+\overline{\partial}^t_{p}:\Omega_c^{p,\bullet}(M)\rightarrow \Omega^{p,\bullet}_c(M)$ where with $\Omega^{p,\bullet}_c(M)$ we mean $\bigoplus_{q=0}^m\Omega_c^{p,q}(M)$. We can define two self-adjoint extensions  of $\overline{\partial}_p+\overline{\partial}^t_p$ taking  
\begin{equation}
\label{twoself}
\overline{\partial}_{p,\max}+\overline{\partial}^t_{p,\min}:L^2\Omega^{p,\bullet}(M,g)\rightarrow L^2\Omega^{p,\bullet}(M,g)
\end{equation}
\begin{equation} 
\label{twoselfs}
\overline{\partial}_{p,\min}+\overline{\partial}^t_{p,\max}:L^2\Omega^{p,\bullet}(M,g)\rightarrow L^2\Omega^{p,\bullet}(M,g)
\end{equation}
where clearly $L^2\Omega^{p,\bullet}(M,g)=\bigoplus_{q=0}^mL^2\Omega^{p,q}(M,g)$. The domain of $\overline{\partial}_{p,\max}+\overline{\partial}^t_{p,\min}$ is given by $\mathcal{D}(\overline{\partial}_{p,\max})\cap\mathcal{D}(\overline{\partial}^t_{p,\min})$ where $\mathcal{D}(\overline{\partial}_{p,\max})=\bigoplus_{q=0}^m\mathcal{D}(\overline{\partial}_{p,q,\max})$ and $\mathcal{D}(\overline{\partial}^t_{p,\min})=\bigoplus_{q=0}^m\mathcal{D}(\overline{\partial}^t_{p,q,\min})$. Analogously the domain of $\overline{\partial}_{p,\min}+\overline{\partial}^t_{p,\max}$ is given by  $\mathcal{D}(\overline{\partial}_{p,\min})\cap\mathcal{D}(\overline{\partial}^t_{p,\max})$ where $\mathcal{D}(\overline{\partial}_{p,\min})=\bigoplus_{q=0}^m\mathcal{D}(\overline{\partial}_{p,q,\min})$ and $\mathcal{D}(\overline{\partial}^t_{p,\max})=\bigoplus_{q=0}^m\mathcal{D}(\overline{\partial}^t_{p,q,\max})$. In particular we have: 
$$\ker(\overline{\partial}_{p,\max/\min}+\overline{\partial}^t_{p,\min/\max})=\bigoplus_{q=0}^m\ker(\overline{\partial}_{p,q,\max/\min})\cap \ker (\overline{\partial}^t_{p,q-1,\min/\max})=\bigoplus_{q=0}^m\mathcal{H}^{p,q}_{\overline{\partial},\text{abs}/\text{rel}}(M,g).$$
Furthermore we recall that the maximal and the minimal $L^2$-$\overline{\partial}$-cohomology of $(M,g)$ are  defined respectively as 
\begin{equation}
\label{chimar}
H^{p,q}_{2,\overline{\partial}_{\max}}(M,g):=\frac{\ker(\overline{\partial}_{p,q,\max})}{\im(\overline{\partial}_{p,q-1,\max})}\ \text{and}\ H^{p,q}_{2,\overline{\partial}_{\min}}(M,g):=\frac{\ker(\overline{\partial}_{p,q,\min})}{\im(\overline{\partial}_{p,q-1,\min})}.
\end{equation}
In particular if $H^{p,q}_{2,\overline{\partial}_{\max}}(M,g)$ is finite dimensional then $\im(\overline{\partial}_{p,q-1,\max})$ is closed and analogously if $H^{p,q}_{2,\overline{\partial}_{\min}}(M,g)$ is finite dimensional then $\im(\overline{\partial}_{p,q-1,\min})$ is closed. If this is the case then we have $H^{p,q}_{2,\overline{\partial}_{\max/\min}}(M,g)\cong \mathcal{H}^{p,q}_{\overline{\partial},\text{abs}/\text{rel}}(M,g).$ Furthermore we recall that if $(M,g)$ is complete  then $\Delta_{\overline{\pa},p,q}:\Omega_c^{p,q}(M)\rightarrow \Omega_c^{p,q}(M)$ and $\overline{\pa}_p+\overline{\pa}_p^t:\Omega_c^{p,\bullet}(M)\rightarrow \Omega_c^{p,\bullet}(M)$ are essentially self-adjoint operators when considered as unbounded and densely defined operators acting on $L^2\Omega^{p,q}(M,g)$ and  $L^2\Omega^{p,\bullet}(M,g)$ respectively. As it is well known this in turn implies that $\overline{\pa}_{p,q,\max}=\overline{\partial}_{p,q,\min}$. Henceforth whenever $(M,g)$ is a complete Hermitian manifold we will simply label   with $\Delta_{\overline{\pa},p,q}:L^2\Omega^{p,q}(M,g)\rightarrow L^2\Omega^{p,q}(M,g)$, $\overline{\pa}_p+\overline{\pa}_p^t:L^2\Omega^{p,\bullet}(M,g)\rightarrow L^2\Omega^{p,\bullet}(M,g)$ and $\overline{\partial}_{p,q}:L^2\Omega^{p,q}(M,g)\rightarrow L^2\Omega^{p,q+1}(M,g)$ the unique closed extension of  $\Delta_{\overline{\pa},p,q}:\Omega_c^{p,q}(M)\rightarrow \Omega_c^{p,q}(M)$,  $\overline{\pa}_p+\overline{\pa}_p^t:\Omega_c^{p,\bullet}(M)\rightarrow \Omega_c^{p,\bullet}(M)$ and  $\overline{\pa}_{p,q}:\Omega_c^{p,q}(M)\rightarrow \Omega_c^{p,q+1}(M)$ respectively. Finally we conclude with a note about the notation that will be used through the paper. When $(p,q)=(0,0)$ we will simply write $\overline{\pa}$, $\overline{\pa}_{\max/\min}$, $\Delta_{\overline{\pa}}$ and $\Delta_{\overline{\pa},\text{rel}/\text{abs}}$ instead of $\overline{\pa}_{0,0}$, $\overline{\pa}_{0,0\max/\min}$, $\Delta_{\overline{\pa},0,0}$ and $\Delta_{\overline{\pa},0,0,\text{rel}/\text{abs}}$ respectively.

\section{Analytic K-homology classes for complex Hermitian spaces}

We start with the following proposition.
\begin{proposition}
\label{bounded}
Let $(M,g)$ be a possibly incomplete Riemannian manifold of dimension $m$.  Let $f\in L^{\infty}(M,g)$ such that $df\in L^{\infty}\Omega^{1}(M,g)$ where $df$ stands for the distributional differential of $f$. Then the following properties hold true:
\begin{itemize}
\item If $\omega\in \mathcal{D}(d_{k,\max})$ then $f\omega\in \mathcal{D}(d_{k,\max})$ and $d_{k,\max}(f\omega)=fd_{k,\max}\omega+df\wedge\omega$. 
\item If $\omega\in \mathcal{D}(d_{k,\min})$ then  $f\omega\in \mathcal{D}(d_{k,\min})$ and $d_{k,\min}(f\omega)=fd_{k,\min}\omega+df\wedge\omega$.
\end{itemize}
Assume  now that  $(M,g)$ is a possibly incomplete complex Hermitian manifold of complex dimension $m$. Let $f\in L^{\infty}(M,g)$ such that $\overline{\partial}f\in L^{\infty}\Omega^{0,1}(M,g)$ where as above $\overline{\partial}f$ is understood in the distributional sense.
Then the following properties hold true:
\begin{itemize}
\item If  $\omega\in \mathcal{D}(\overline{\partial}_{0,q,\max})$  then  $f\omega\in \mathcal{D}(\overline{\partial}_{0,q,\max})$ and $\overline{\partial}_{0,q,\max}(f\omega)=f\overline{\partial}_{0,q,\max}\omega+\overline{\partial}f\wedge\omega$.
\item $\omega\in \mathcal{D}(\overline{\partial}_{0,q,\min})$  then  $f\omega\in \mathcal{D}(\overline{\partial}_{0,q,\min})$ and $\overline{\partial}_{0,q,\min}(f\omega)=f\overline{\partial}_{0,q,\min}\omega+\overline{\partial}f\wedge\omega$.
\end{itemize}
Finally completely analogous results hold if we replace $d_k$ with $d_k^t$ and in the complex case $\overline{\partial}_{0,q}$ with $\overline{\partial}_{0,q}^t$ where $d_k^t$ and $\overline{\partial}_{0,q}^t$ are respectively the formal adjoint of $d_k$ and $\overline{\partial}_{0,q}$.
\end{proposition}

\begin{proof}
The first two statements above are a particular case of \cite{SympBei} Prop. 2.3. In the complex setting, the corresponding  statements  for $\overline{\partial}$,  follow by applying the same strategy, with the obvious modifications, that is used in the proof of Prop. 2.3 in \cite{SympBei}. Finally again the same strategy can be used to give  a proof if we replace $d_k$ with $d_k^t$ and, in the complex case, $\overline{\partial}_{0,q}$ with $\overline{\partial}_{0,q}^t$.
\end{proof}

We proceed very briefly by recalling some  notions of complex analytic geometry. Complex spaces are a classical topic in complex geometry and we refer to \cite{Fischer} and \cite{GraRe} for definitions and properties. Consider now a reduced complex space $X$ and let $\mathcal{O}_X$ be the sheaf of holomorphic functions on $X$. The sheaf of weakly holomorphic function on $X$, labeled by $\tilde{\mathcal{O}}_X$, is the sheaf the assigns to each open subset $U$ of $X$ the space of functions $f:\reg(U)\rightarrow \mathbb{C}$ that are  locally bounded on $U$ and holomorphic on $\reg(U)$. A point $p\in X$ is said normal if $\tilde{\mathcal{O}}_{X,p}=\mathcal{O}_{X,p}$. $X$ is normal if $\tilde{\mathcal{O}}_{X,p}=\mathcal{O}_{X,p}$ for any $p\in X$. In this case $\sing(X)$ has complex codimension at least 2.
If $X$ is not normal then there exists a  normalization of $X$, $\nu: \tilde{X}\rightarrow X$. Here we skip the definition and we refer to \cite{GraRe} and \cite{Fischer} for  precise statements. For our purpose it suffices to recall that $\tilde{X}$ is a normal complex space, $\nu:\tilde{X}\rightarrow X$ is a finite and surjective holomorphic map such that $\nu_*\mathcal{O}_{\tilde{X}}=\tilde{\mathcal{O}}_X$ and $\nu|_A:A\rightarrow B$ is a biholomorphism, where $A$ is open and dense in $\tilde{X}$, $B$ is open and dense in $X$ and $X\setminus B$ is the subset of $X$  made by non-normal points. Moreover  we recall that an irreducible complex space $X$ is a reduced complex space such that $\reg(X)$, the regular part of $X$, is connected.\\ A paracompact  and reduced complex space $X$ is said \emph{Hermitian} if  the regular part of $X$ carries a Hermitian metric $h$  such that for every point $p\in X$ there exists an open neighborhood $U\ni p$ in $X$, a proper holomorphic embedding of $U$ into a polydisc $\phi: U \rightarrow \mathbb{D}^N\subset \mathbb{C}^N$ and a Hermitian metric $g$ on $\mathbb{D}^N$ such that $(\phi|_{\reg(U)})^*g=h$, see for instance \cite{Takeo} or \cite{JRupp}. In this case we will write $(X,h)$ and with a little abuse of language we will say that $h$ is a \emph{Hermitian metric on $X$}. Natural examples of Hermitian complex spaces are provided by  analytic sub-varieties of  complex Hermitian manifolds endowed with the metric induced by the Hermitian  metric of the ambient space. In particular, within this class of examples, we have any complex projective variety $V\subset \mathbb{C}\mathbb{P}^n$ endowed  with the K\"ahler metric induced by the Fubini-Study metric of $\mathbb{C}\mathbb{P}^n$. We point out explicitly that all the Hermitian metrics on $X$ belong to the same quasi-isometry class. This follows easily by the lifting lemma, see \cite{SingDef} Remark 1.30.1 page 37.\\ Moreover, in order to state the next results, we spend a few words concerning  resolution of singularities. We refer to the celebrated work of Hironaka \cite{Hiro} and to \cite{BieMil} for a thorough discussion  on this subject. Here we simply recall what is strictly necessary for our purposes.\\ Let $X$ be a compact irreducible complex space. Then there exists a compact complex manifold $M$, a divisor with only normal crossings $D\subset M$ and a surjective holomorphic map $\pi:M\rightarrow X$ such that $\pi^{-1}(\sing(X))=D$ and 
\begin{equation}
\label{hiro}
\pi|_{M\setminus D}: M\setminus D\longrightarrow X\setminus \sing(X)
\end{equation}
is a biholomorphism.  We have now the following definition.

\begin{definition}
\label{smooth}
Let $X$ be a compact and irreducible complex space. Let $f:X\rightarrow \mathbb{C}$ be a continuous function. We will say that $f$ is smooth if for any point $p\in X$ there exists an open neighborhood $U$ of $p$, a holomorphic embedding of $\phi:U\rightarrow \mathbb{C}^N$ for some $N$, an open subset $W\subset \mathbb{C}^N$ with $\phi(U)\subset W$ and a smooth function $\beta:W \rightarrow \mathbb{C}^N$ such that $\beta\circ \phi=f|_U$.  
\end{definition}
We will label with $S(X)$ the set of smooth functions on $X$. Clearly if $f\in S(X)$ then $f|_{\reg(X)}\in C^{\infty}(\reg(X))$.

\begin{proposition}
\label{dense}
Let $X$ be a compact and irreducible complex space. Then $S(X)$ is a dense $*$-subalgebra of $C(X)$.
\end{proposition}

\begin{proof}
Clearly $A(X)\neq \emptyset$ because $\mathbb{C}\subset S(X)$. Moreover it is also clear that $S(X)$ is a $*$-subalgebra of $C(X)$. In order to conclude that $S(X)$ is dense in $C(X)$ we want to use the Stone-Weierstraas theorem. Therefore we are left to prove that given two points $p,q\in X$ with $p\neq q$ there exists a function $f\in S(X)$ such that $f(p)\neq f(q)$. Let $U$ be an open neighborhood of $p$ such that $q\notin \overline{U}$ and such that there exists a holomorphic embedding $\phi:U\rightarrow \mathbb{C}^N$ for some $N$. Let $V$ be another open neighborhood of $p$ such that $\overline{V}\subset U$. Let $A\subset \mathbb{C}^N$ be an open subset such that $A\cap \phi(U)=\phi(V)$.
Let $\alpha\in C^{\infty}_c(A)$ such that $\alpha(\phi(p))=1$. Let $f:=\alpha\circ \phi$ and let 
$$\label{pppp}
\tilde{f}:= \left\{
\begin{array}{ll}
f\ \text{on}\ V\\
0\ \text{on}\ U\setminus V
\end{array}
\right.
$$ 
By construction it is clear that  $f\in S(X)$ and that  $f(p)=1$ and $f(q)=0$. The proof is thus complete.
\end{proof}

\begin{proposition}
\label{bbb}
Let $(X,h)$ be a compact and irreducible Hermitian complex space. Let $f\in S(X)$. Then $d(f|_{\reg(X)})\in L^{\infty}\Omega^1(\reg(X),h)$ and  $\overline{\pa}(f|_{\reg(X)})\in L^{\infty}\Omega^{0,1}(\reg(X),h)$.
\end{proposition}

\begin{proof}
Clearly it is enough to show that $d(f|_{\reg(X)})\in L^{\infty}\Omega^1(\reg(X),h)$. Indeed the other property follows immediately by the fact that $d(f|_{\reg(X)})=\partial(f|_{\reg(X)})+\overline{\partial}(f|_{\reg(X)})$.
 Moreover, since $X$ is compact, it is enough to show that for each $p\in \sing(X)$ there exists an open neighborhood $U$ such that $d(f|_{\reg(U)})\in L^{\infty}\Omega^{1}(\reg(U),h|_{\reg(U)})$. Since $f\in S(X)$ we can find an open neighborhood $U$ of $p$, a holomorphic embedding $\phi:U\rightarrow \mathbb{C}^N$, a open subset $W\subset \mathbb{C}^N$ with $\phi(U)\subset W$, a Hermitian metric $g$ on $W$ and a smooth function $\beta\in C^{\infty}(W)$ such that $\phi^*g=h|_{\reg(U)}$ and $\beta\circ \phi=f|_U$. Let $V$ be another open neighborhood of $p$ such that $\overline{V}$ is compact and $\overline{V}\subset U$. Let $A\subset W$ be another open subset of $\mathbb{C}^N$ such that $\overline{A}$ is compact, $\overline{A}\subset W$ and $\phi(V)\subset A$. Since $\overline{A}$ is compact we have $d(\beta|_A)\in L^{\infty}\Omega^1(A,g|_A)$ and, by arguing as in \cite{Bei2017} Prop. 1.2, we get $d(\beta|_{\reg(\phi(V))})\in L^{\infty}\Omega^1(\reg(\phi(V)),i^*_{\reg(\phi(V))}g)$ where $i_{\reg(\phi(V))}:\reg(\phi(V))\rightarrow A$ is the canonical inclusion. Finally this implies immediately that $d(f|_{\reg(V)})\in L^{\infty}\Omega^1(\reg(V),h|_{\reg(V)})$ because $\phi^*g=h|_{\reg(U)}$, $\beta\circ \phi=f|_U$ and $\overline{V}\subset U$.
\end{proof}

When $\dim(\sing(X))=0$ it is convenient to replace $S(X)$ with $S_c(X)\subset S(X)$ which is defined as follows:

\begin{align}
\label{subalgebra}
& S_c(X):=\{f\in C(X)\cap C^{\infty}(\reg(X))\ \text{such that for each}\ p\in \sing(X)\ \text{there exists an}\\ 
& \nonumber \text{open   neighborhood}\ U\ \text{of}\ p\ \text{with}\ f|_U=c\in \mathbb{C}\}.
\end{align}

It is immediate to check that  $S_c(X)$ is dense in $C(X)$. Moreover we have  the following useful proposition which improves, in the setting of isolated singularities, the conclusion of  Prop. \ref{bounded}:

\begin{proposition}
\label{stable}
Let $(X,h)$ be a compact and irreducible Hermitian complex space such that $\dim(\sing(X))=0$. Let $\overline{D}_{p,q}:L^2\Omega^{p,q}(\reg(X),h)\rightarrow L^2\Omega^{p,q+1}(\reg(X),h)$ be any closed extension of $\overline{\partial}_{p,q}:\Omega^{p,q}_c(\reg(X))\rightarrow \Omega_c^{p,q+1}(\reg(X))$. Then, for any $\omega\in \mathcal{D}(\overline{D}_{p,q})$ and  $f\in S_c(X)$ we have $f\omega\in \mathcal{D}(\overline{D}_{p,q})$. Furthermore the same conclusion holds true for any arbitrary closed extension $\overline{D}_{p,q}^t:L^2\Omega^{p,q+1}(\reg(X),h)\rightarrow L^2\Omega^{p,q}(\reg(X),h)$, $D_k:L^2\Omega^{k}(\reg(X),h)\rightarrow L^2\Omega^{k+1}(\reg(X),h)$ and $D_k^t:L^2\Omega^{k+1}(\reg(X),h)\rightarrow L^2\Omega^{k}(\reg(X),h)$ of $\overline{\partial}_{p,q}^t:\Omega^{p,q+1}_c(\reg(X))\rightarrow \Omega_c^{p,q}(\reg(X))$, $d_k:\Omega^{k}_c(\reg(X))\rightarrow \Omega_c^{k+1}(\reg(X))$ and $d_k^t:\Omega^{k+1}_c(\reg(X))\rightarrow \Omega_c^{k}(\reg(X))$, respectively.
\end{proposition}
\begin{proof}
We give the proof assuming that $\sing(X)$ is made only by one isolated  singular point. The general case follows by the same strategy with the obvious modifications. First we start with the following considerations. Let $\omega\in \mathcal{D}(\overline{D}_{p,q})$ smooth and let $\eta\in \mathcal{D}(\overline{D}_{p,q}^*)$ smooth too. By the very definition of adjoint operator and the fact that both $\omega$ and $\eta$ are smooth we have 
\begin{equation}
\label{adjoint}
\langle\overline{\partial}_{p,q}\omega,\eta\rangle_{L^2\Omega^{p,q+1}(\reg(X),h)}=\langle\omega,\overline{\partial}^t_{p,q}\eta\rangle_{L^2\Omega^{p,q}(\reg(X),h)}
\end{equation}
Let $c:\Lambda^{p,q}(\reg(X))\rightarrow \Lambda^{q,p}(\reg(X))$ and $*:\Lambda^{p,q}(\reg(X))\rightarrow \Lambda^{m-q,m-p}(\reg(X))$ be the conjugation and the Hodge star operator, respectively. Let $\psi:=c(*\eta)$. Then we can rewrite the left-hand side of \eqref{adjoint} as $\int_{\reg(X)}\overline{\partial}_{p,q}\omega\wedge\psi$. Let $k=p+q$. Keeping in mind that $*^2=(-1)^{k(2m-k)}$ and that $\overline{\partial}_{p,q}^t=-*\partial_{m-q-1,m-p}*$ we can rewrite the right-hand side of \eqref{adjoint} as $(-1)^{k(2m-k)+1}\int_{\reg(X)}\omega\wedge\overline{\partial}_{m-p,m-q-1}\psi$$=(-1)^{k+1}\int_{\reg(X)}\omega\wedge\overline{\partial}_{m-p,m-q-1}\psi$ since $k+1$ is even if and only if $k(2m-k)+1$ is. 
Consider now an exhaustion of $\reg(X)$ made by relatively compact open subsets $\{A_n\}$ with smooth boundary. Using the fact that $\partial_{p,q}\omega\wedge\psi$ is identically zero we have 
\begin{align}
\nonumber & \int_{\reg(X)}\overline{\partial}_{p,q}\omega\wedge\psi=\int_{\reg(X)}d_k\omega\wedge\psi=\lim_{n\rightarrow \infty}\int_{A_n}d_k\omega\wedge\psi=\\
&\nonumber   \lim_{n\rightarrow \infty}\left(\int_{A_n}d_{2m-1}(\omega\wedge\psi) +(-1)^{k+1}\int_{A_n}\omega\wedge d_{2m-k-1}\psi\right)=\\
&\nonumber \lim_{n\rightarrow \infty}\left(\int_{\partial A_n}i^*(\omega\wedge\psi) +(-1)^{k+1}\int_{A_n}\omega\wedge \overline{\partial}_{m-p,m-q-1}\psi\right).
\end{align}
where $i:\partial A_n\hookrightarrow \reg(X)$ is the inclusion. Since $$\lim_{n\rightarrow\infty}\int_{A_n}\omega\wedge \overline{\partial}_{m-p,m-q-1}\psi=\int_{\reg(X)}\omega\wedge \overline{\partial}_{m-p,m-q-1}\psi$$  we can conclude that
\begin{equation}
\label{vanishes}
\lim_{n\rightarrow \infty}\int_{\partial A_n}i^*(\omega\wedge\psi)=0.
\end{equation}
Consider now $\overline{\partial}_{p,q}(f\omega)$. We have
\begin{align}
 \nonumber & \langle\overline{\partial}_{p,q}(f\omega),\eta\rangle_{L^2\Omega^{p,q+1}(\reg(X),h)}=\int_{\reg(X)}\overline{\partial}_{p,q}(f\omega)\wedge\psi=\int_{\reg(X)}d_k(f\omega)\wedge\psi=\\
\nonumber &\lim_{n\rightarrow \infty}\int_{A_n}d_k(f\omega)\wedge\psi= \lim_{n\rightarrow \infty}\left(\int_{A_n}d_{2m-1}(f\omega\wedge\psi) +(-1)^{k+1}\int_{A_n}f\omega\wedge d_{2m-k-1}\psi\right)=\\
&\nonumber \lim_{n\rightarrow \infty}\left(\int_{\partial A_n}i^*(f\omega\wedge\psi) +(-1)^{k+1}\int_{A_n}f\omega\wedge \overline{\partial}_{m-p,m-q-1}\psi\right).
\end{align} 
Clearly $$\lim_{n\rightarrow \infty}(-1)^{k+1}\int_{A_n}f\omega\wedge \overline{\partial}_{m-p,m-q-1}\psi=(-1)^{k+1}\int_{\reg(X)}f\omega\wedge \overline{\partial}_{m-p,m-q-1}\psi=\langle f\omega,\overline{D}_{p,q}^*\eta\rangle_{L^2\Omega^{p,q}(\reg(X),h)}.$$ Moreover, as $f\in S_c(X)$, there exists an integer $\overline{n}$ such that for any $n\geq \overline{n}$ we have $f|_{\partial A_n}=\ell$ for some constant $\ell\in \mathbb{R}$ independent on $n$. This tells us that 
$$ \lim_{n\rightarrow \infty}\int_{\partial A_n}i^*(f\omega\wedge\psi)= \lim_{n\rightarrow \infty}\ell\int_{\partial A_n}i^*(\omega\wedge\psi)=0$$ thanks to \eqref{vanishes}. Summarizing we showed that for any $\omega\in \mathcal{D}(\overline{D}_{p,q})\cap\Omega^{p,q}(\reg(X))$,  $f\in S_c(X)$ and $\eta\in \mathcal{D}(\overline{D}_{p,q}^*)\cap\Omega^{p,q+1}(\reg(X))$ we have $$\langle\overline{\partial}_{p,q}(f\omega),\eta\rangle_{L^2\Omega^{p,q+1}(\reg(X),h)}=\langle f\omega,\overline{D}_{p,q}^*\eta\rangle_{L^2\Omega^{p,q}(\reg(X),h)}.$$
As $\mathcal{D}(\overline{D}_{p,q}^*)\cap\Omega^{p,q+1}(\reg(X))$ is dense in  $\mathcal{D}(\overline{D}_{p,q}^*)$  with respect to the corresponding graph norm we can conclude that $f\omega\in \mathcal{D}(\overline{D}_{p,q})$. Now, if we consider an arbitrary $\omega\in \mathcal{D}(\overline{D}_{p,q})$, it is enough to observe that $\mathcal{D}(\overline{D}_{p,q})\cap\Omega^{p,q}(\reg(X))$ is  dense in $\mathcal{D}(\overline{D}_{p,q})$  with respect to the corresponding graph norm and that, given  a sequence $\{\omega_j\}\subset \mathcal{D}(\overline{D}_{p,q})\cap \Omega^{p,q}(\reg(X))$ converging to $\omega$,  then also $\{f\omega_j\}$ converges to $f\omega$ in the graph norm. Finally the analogous statements for $d_k$, $d^t_k$ and $\overline{\partial}^t_{p,q}$ follow by applying the same strategy.
\end{proof}

We recall now the definition of $KK_0(C(X),\mathbb{C})$. For more details we refer to \cite{Kasparov}, \cite{BaHiSc} and the references cited there. Given the $C^*$-algebra $C(X)$ an even Fredholm module is a triple $(H,\rho, F)$ satisfying the following properties:
\begin{itemize}
\item $H$ is a separable Hilbert space,
\item $\rho$ is a  representation $\rho:C(X)\rightarrow \mathcal{B}(H)$ of $C(X)$ as bounded operators on $H$
\item $F$ is an operator on $H$ such that for all $f\in C(X)$:
$$(F^2-\Id)\circ \rho(f),\ (F-F^*)\circ \rho(f)\ \text{and}\  [F,\rho(f)]\ \text{lie in}\ \mathcal{K}(H)$$
where $\mathcal{K}(H)\subset \mathcal{B}(H)$ is the space of compact operators. 
\item The Hilbert space $H$ is equipped with a $\mathbb{Z}_2$-grading $H=H^+\oplus H^- $ in such a
way that for each $f\in C(X)$, the operator $\rho(f)$ is even-graded, while the operator $F$ is
odd-graded.
\end{itemize}
Let $(H_1, \rho_1, F_1)$ and $(H_2, \rho_2, F_2)$ be even Fredholm modules over $C(X)$. A {\em unitary equivalence} between them is a  grading-zero unitary isomorphism $u:H_1\rightarrow H_2$ which intertwines the representations $\rho_1$ and $\rho_2$ and the operators $F_1$ and  $F_2$.\\
Given two even Fredholm modules $(H,\rho,F_0)$ and $(H,\rho,F_1)$ an {\em operator homotopy} between them is a family of  Fredholm modules $(H,\rho,F_t)$ parameterized by $t\in [0, 1]$ in such a way that the representation $\rho$, the Hilbert space $H$ and its grading structures remain constant but the operator $F_t$ varies with $t$ and the function $[0,1]\rightarrow \mathcal{B}(H)$, $t\mapsto F_t$ is norm continuous. In this case we will say that $(H,\rho,F_0)$ and $(H,\rho,F_1)$ are (operator) homotopic.\\
Clearly we can define in a natural way the notion  of direct sum for Fredholm modules: one takes the direct sum
of the Hilbert spaces, of the representations, and of the operators $F$. The zero module has
zero Hilbert space, zero representation, and zero operator.\\
Now we can give Kasparov's definition of $K$-homology. The $K$-homology group $KK_0(C(X),\mathbb{C})$ is the abelian group with one generator $[x]$ for each
unitary equivalence class of  even Fredholm modules over $C(X)$ and with the following
relations:
\begin{itemize}
\item if $x_0$ and $x_1$ are operator homotopic even Fredholm modules then $[x_0]=[x_1]$ in $KK_0(C(X),\mathbb{C})$,
\item if $x_0$ and $x_1$ are any two even Fredholm modules then $[x_0+x_1]=[x_0]+[x_1]$ in $KK_0(C(X),\mathbb{C})$.
\end{itemize}
Now we go on by recalling the notion of  {\em even unbounded Fredholm module} for the $C^*$-algebra $C(X)$. This is a triple
$(H,\upsilon, D)$ such that:
\begin{itemize}
\item $H$ is a Hilbert space endowed with a unitary $*$-representation $\upsilon:C(X)\rightarrow \mathcal{B}(H)$; $D$ is a self-
adjoint unbounded linear operator on $H$;
\item there is a dense $*$-subalgebra $A\subset C(X)$ such that  for all $a\in A$ the domain of $D$ is invariant
by $a$ and $[D,a]$ extends to a bounded operator on $H$;
\item $\upsilon(a)(1+D^2)^{-1}$ is a compact operator on $H$ for any $a\in A$;
\item $H$ is equipped with a grading $\tau=\tau^*$, $\tau^2=Id$, such that $\tau\circ \upsilon=\upsilon\circ \tau$ and $\tau\circ D = -\tau\circ D$. In other words $\tau$ commutes with $\upsilon$ and anti-commutes with $D$.
\end{itemize}
An odd unbounded Fredholm module is defined omitting the last condition. We have now the following important result which is a particular case of \cite{Baaj-Julg}, Prop 2.2:
\begin{proposition}
Let $(H,\upsilon,D)$ be an even unbounded Fredholm module for $C(X)$. Then $(H,\upsilon,D\circ (\Id+D^2)^{-1/2})$ is an even bounded Fredholm module for $C(X)$.
\end{proposition}
\noindent In what follow, given an unbounded Fredholm module as above, with the notation $[D]$ we will mean the class induced by $H$, $\upsilon$ and $D\circ (\Id+D^2)^{-1/2}$ in $KK_0(C(X),\mathbb{C})$.
After this concise reminder on analytic $K$-homology we continue with the following proposition. It is concerned with unbounded Fredholm modules in the setting of Hermitian complex spaces.

\begin{proposition}
\label{minclass}
Let $(X,h)$ be a compact and irreducible Hermitian complex space of complex dimension $v$. Assume that $\sing(X)$ is made of isolated points. Then the operator
\begin{equation}
\label{hdmin}
\overline{\partial}_{0,\min}+\overline{\partial}_{0,\max}^t:L^2\Omega^{0,\bullet}(\reg(X),h)\rightarrow L^2\Omega^{0,\bullet}(\reg(X),h)
\end{equation}
 defines an  unbounded Fredholm module for $C(X)$ and thus a class 
 \begin{equation}\label{eq:classes}
 [\overline{\partial}_{0,\min}+\overline{\partial}_{0,\max}^t]\in KK_0(C(X),\mathbb{C})
 \end{equation}
  Moreover this  class does not depend on the particular Hermitian metric on $\reg(X)$ that we fix within the quasi-isometry class of $h$. In particular it does not depend on the particular Hermitian metric that we fix on $X$.
\end{proposition}

\smallskip
\noindent
{\bf Notation.} We set $ \overline{\eth}_{\operatorname{rel}}:=
 \overline{\partial}_{0,\min}+\overline{\partial}_{0,\max}^t$.
 
\begin{proof}
 We take $H=L^2\Omega^{0,\bullet}(\reg(X),h)$ and the representation is the one given by pointwise multiplication. As a dense $*$-subalgebra of $C(X)$ we consider $S_c(X)$, see \eqref{subalgebra}.  Now let us consider a function $f\in S_c(X)$. The domain of $\overline{\eth}_{\operatorname{rel}}$ is given by $\mathcal{D}(\overline{\eth}_{\operatorname{rel}})=\bigoplus_q(\mathcal{D}(\overline{\partial}_{0,q,\min})\cap \mathcal{D}(\overline{\partial}^t_{0,q,\max}))$. Thus if we consider an element $\omega\in \mathcal{D}(\overline{\eth}_{\operatorname{rel}})$ then, by Prop. \ref{bounded},  we can conclude that $f\omega\in \mathcal{D}(\overline{\eth}_{\operatorname{rel}})$. Moreover we have $[\overline{\eth}_{\operatorname{rel}},f]\omega=\overline{\partial}f\wedge \omega-(\overline{\partial}f\wedge)^*\omega$ where $(\overline{\partial}f\wedge)^*$ is the adjoint of the map $\eta\mapsto \overline{\partial}f\wedge\eta$. As $\overline{\pa}f\in \Omega_c^{0,1}(\reg(X))$ we can conclude that $[\overline{\eth}_{\operatorname{rel}},f]$ induces a bounded operator on $L^2\Omega^{0,\bullet}(\reg(X),h)$.
As we assumed $\dim(\sing(X))=0$ we can use \cite{FraBei} Cor. 5.2 to conclude that $\overline{\eth}_{\operatorname{rel}}:L^2\Omega^{0,\bullet}(\reg(X),h)\rightarrow L^2\Omega^{0,\bullet}(\reg(X),h)$ has discrete spectrum and this is well known to be equivalent to the compactness of $(\overline{\eth}_{\operatorname{rel}}^2+1)^{-1}:L^2\Omega^{0,\bullet}(\reg(X),h)\rightarrow L^2\Omega^{0,\bullet}(\reg(X),h)$. Finally, arguing as in \cite{Hilsum-LNM} ,  we can prove that $[\overline{\eth}_{\operatorname{rel}}]$ does not depend on the particular Hermitian metric on $\reg(X)$ that we fix within the quasi-isometry class of $h$. In particular $[\overline{\eth}_{\operatorname{rel}}]$ does not  depend on the particular Hermitian metric that we fix on $X$. The proof is thus complete.
\end{proof}

Analogously we can associate an unbounded Fredholm module also to  
the operator $\overline{\pa}_{0,\max}+\overline{\pa}_{0,\min}^t$. This is indeed  the goal of the next proposition.
\begin{proposition}
\label{maxclass}
Let $(X,h)$ be a compact and irreducible Hermitian complex space of complex dimension $v$. Assume that $\sing(X)$ is made of isolated points. Then the operator
\begin{equation}
\label{hdmin}
\overline{\partial}_{0,\max}+\overline{\partial}_{0,\min}^t:L^2\Omega^{0,\bullet}(\reg(X),h)\rightarrow L^2\Omega^{0,\bullet}(\reg(X),h)
\end{equation}
 defines an  unbounded Fredholm module for $C(X)$ and thus a class 
 \begin{equation}\label{eq:classes}
 [\overline{\eth}_{\operatorname{abs}}]:=[\overline{\partial}_{0,\max}+\overline{\partial}_{0,\min}^t]\in KK_0(C(X),\mathbb{C})
 \end{equation}
  Moreover this  class does not depend on the particular Hermitian metric on $\reg(X)$ that we fix within the quasi-isometry class of $h$. In particular it does not depend on the particular Hermitian metric that we fix on $X$.
\end{proposition}

\begin{proof}
The proof is completely analogous to the proof of Prop. \ref{minclass}. 
The only part that  requires to be pointed out is the fact that  $\overline{\partial}_{0,\max}+\overline{\partial}_{0,\min}^t$ has compact resolvent. According to \cite{JRupp} Th. 1.9 we know that $H^{0,q}_{2,\overline{\partial}_{\max}}(\reg(X),h)$ is finite dimensional. Therefore, by \cite{BruLe} Th. 2.4, we know that $\overline{\partial}_{0,\max}+\overline{\partial}_{0,\min}^t:L^2\Omega^{0,\bullet}(\reg(X),h)\rightarrow L^2\Omega^{0,\bullet}(\reg(X),h)$ is a Fredholm operator on its domain endowed with the graph norm. This in turn tells us that $\overline{\partial}_{0,\max}+\overline{\partial}_{0,\min}^t:L^2\Omega^{0,\bullet}(\reg(X),h)\rightarrow L^2\Omega^{0,\bullet}(\reg(X),h)$ has compact resolvent if and only if the following inclusion is a compact operator 
\begin{equation}
\label{greenop}
\mathcal{D}(\overline{\partial}_{0,\max}+\overline{\partial}_{0,\min}^t)\cap \im(\overline{\partial}_{0,\max}+\overline{\partial}_{0,\min}^t)\hookrightarrow L^2\Omega^{0,\bullet}(\reg(X),h)
\end{equation}
where the space on the left hand side of the inclusion is endowed with the graph norm of $\overline{\partial}_{0,\max}+\overline{\partial}_{0,\min}^t$. Finally, since we required $\dim(\sing(X))=0$, we can deduce immediately  \eqref{greenop} by  \cite{OvRu} Th. 1.2. This concludes the proof.
\end{proof}

The last goal of this section is to show that in the setting of compact and irreducible Hermitian complex spaces of complex dimension 2 we can prove Prop. \ref{minclass} without any assumption on $\sing(X)$.   

\begin{proposition}
\label{minclasssur}
Let $(X,h)$ be a   compact and irreducible Hermitian complex space of complex dimension 2. Then the operator
\begin{equation}
\label{hdsur}
\overline{\partial}_{0,\min}+\overline{\partial}_{0,\max}^t:L^2\Omega^{0,\bullet}(\reg(X),h)\rightarrow L^2\Omega^{0,\bullet}(\reg(X),h)
\end{equation}
 defines an   unbounded Fredholm module for $C(X)$ and thus a class 
 \begin{equation}\label{eq:classsurfaces}
 [\overline{\eth}_{\operatorname{rel}}]:=[\overline{\partial}_{0,\min}+\overline{\partial}_{0,\max}^t]\in KK_*(C(X),\mathbb{C})
 \end{equation}
 Moreover this class does not depend on the particular Hermitian metric that we fix within the quasi-isometry class of $h$. In particular it  does not depend on the particular Hermitian metric that we fix on $X$.
\end{proposition}

\begin{proof}
As usual we take we take $H=L^2\Omega^{0,\bullet}(\reg(X),h)$ and the representation is the one given by pointwise multiplication. We take $A$ as the $*$-subalgebra of $C(X)$ given by $S(X)$, see Def. \ref{smooth}. By Prop. \ref{dense} we know that $A$ is dense in $C(X)$. Now let us consider a function $f\in A$. The domain of $\overline{\eth}_{\operatorname{rel}}$ is given by $\mathcal{D}(\overline{\eth}_{\operatorname{rel}})=\bigoplus_q(\mathcal{D}(\overline{\partial}_{0,q,\min})\cap \mathcal{D}(\overline{\partial}^t_{0,q,\max}))$. Thus if we consider an element $\omega\in \mathcal{D}(\overline{\eth}_{\operatorname{rel}})$ then, by Prop. \ref{bounded} and Prop. \ref{bbb} we can conclude that $f\omega\in \mathcal{D}(\overline{\eth}_{\operatorname{rel}})$. Moreover we have $[\overline{\eth}_{\operatorname{rel}},f]\omega=\overline{\partial}f\wedge \omega-(\overline{\partial}f\wedge)^* \omega$ where, as previously explained, $(\overline{\partial}f\wedge)^*$ is the adjoint of the map $\eta\mapsto \overline{\partial}f\wedge\eta$. Therefore, according to Prop. \ref{bbb}, we can conclude that $[\overline{\eth}_{\operatorname{rel}},f]$ induces a bounded operator on $L^2\Omega^{0,\bullet}(\reg(V),h)$ because $\overline{\partial}f\in L^{\infty}\Omega^{0,1}(\reg(X),h)$.
By  \cite{Bei2017}  we know that  $(\overline{\eth}_{\operatorname{rel}}^2+1)^{-1}:L^2\Omega^{0,\bullet}(\reg(X),h)\rightarrow L^2\Omega^{0,\bullet}(\reg(X),h)$ is a compact operator. Finally, arguing as in \cite{Hilsum-LNM},  we can prove that $[\overline{\eth}_{\operatorname{rel}}]$ does not depend on the particular Hermitian metric that we fix within the quasi-isometry class of $h$. In particular $[\overline{\eth}_{\operatorname{rel}}]$ does not  depend on the particular Hermitian metric that we fix on $X$. The proof is thus complete.
\end{proof}
%

\section{Resolutions and K-homology classes}

Let $X$ be a compact complex space of complex dimension $m$. According to the celebrated Hironaka's theorem on resolution of singularities there exists a compact complex manifold $M$ a divisor with only normal crossings $D\subset M$ and a  surjective holomorphic map $\pi:M\rightarrow X$ such that $\pi^{-1}(\sing(X))= D$ and $$\pi|_{M\setminus D}:M\setminus D\rightarrow \reg(X)$$ is a biholomorphism. Let us fix an arbitrary  Hermitian metric $g$ on $M$. Let $\overline{\pa}_{0,q}^t:\Omega^{0,q+1}(M)\rightarrow \Omega^{0,q}(M)$ be the formal adjoint of $\overline{\pa}_{0,q}:\Omega^{0,q}(M)\rightarrow \Omega^{0,q+1}(M)$. Since $M$ is compact and 
\begin{equation}
\label{hodgedol}
\overline{\pa}_0+\overline{\pa}^t_0:\Omega^{0,\bullet}(M)\rightarrow \Omega^{0,\bullet}(M)
\end{equation}
is elliptic we know that \eqref{hodgedol} is essentially self-adjoint when we look at it as an unbounded and densely defined operator acting on $L^2\Omega^{0,\bullet}(M,g)$. We label this unique (and therefore self-adjoint) extension by $\overline{\eth}:L^2\Omega^{0,\bullet}(M,g)\rightarrow L^2\Omega^{0,\bullet}(M,g)$. Moreover it is well known that the pair $(L^2\Omega^{0,\bullet}(M,g), \overline{\eth})$ defines a class in $KK_0(C(M),\mathbb{C})$ that does not depend on the particular Hermitian metric that we fix on $M$. We label this class by $[\overline{\eth}_M]$.

As we have previously seen, in the setting of  compact and irreducible Hermitian complex spaces with either $\dim(\sing(X))=0$ or $\dim(X)=2$ we have  a K-homology class labeled by $[\overline{\eth}_{\operatorname{rel}}]$. Since the analytic $K$-homology is covariant we get, through the map $\pi$, a morphism $\pi_*:KK_0(C(M),\bbC)\rightarrow KK_0(C(X),\bbC)$. According to the results proved in \cite{PardonSternJAMS} and \cite{JRupp} it seems a natural  problem to compare $[\overline{\eth}_{\operatorname{rel}}]$ with $\pi_*[\overline{\eth}_M]$. This is the aim of the next result in the case $\dim(\sing(X))=0$.

\begin{theorem}
\label{Kequalities}
Let $(X,h)$ be  a compact and irreducible Hermitian complex space with only isolated singularities. Then we have the following equality in $KK_0(C(X),\bbC)$:
\begin{equation}
\label{firstequality}
\pi_*[\overline{\eth}_M]
=[\overline{\eth}_{\operatorname{rel}}].
\end{equation}
\end{theorem}

\begin{proof}

In order to prove the theorem we shall need the following

\begin{proposition}
\label{othermodel}
Let $X$ be a compact and irreducible complex space of complex  dimension $m$ such that $\dim(\sing(X))=0$. Let $\pi:M\rightarrow X$ be a  resolution of $X$. Let $\rho$ be a  Hermitian metric on $\reg(X)$ such  that $\dim (H^{0,q}_{2,\overline{\pa}_{\min}}(\reg(X),\rho))<\infty$ for each $q$.
Assume that
$$\sum_q (-1)^q \dim (H^{0,q}_{2,\overline{\pa}_{\min}}(\reg(X),\rho))= 
\sum_q (-1)^q  \dim (H^{0,q}_{\overline{\pa}}(M))\,.$$ Consider the operator $$\overline{\pa}_{0,\min}+\overline{\pa}_{0,\max}^t:L^2\Omega^{0,\bullet}(\reg(X),\rho)\rightarrow L^2\Omega^{0,\bullet}(\reg(X),\rho)$$ and let 
$$ T:L^2\Omega^{0,\bullet}(\reg(X),\rho)\rightarrow L^2\Omega^{0,\bullet}(\reg(X),\rho)$$ be the bounded operator
defined as $T:= (\overline{\partial}_{0,\min} +  \overline{\partial}_{0,\max}^t ) \circ L^{-1/2}$ where 
$$L:=\Pi_{\Ker} + (\overline{\partial}_{0,\min} +  \overline{\partial}_{0,\max}^t)^2$$
and  with $\Pi_{\Ker}$ denoting the orthogonal projection onto the null space of $(\overline{\partial}_{0,\min} +  \overline{\partial}_{0,\max}^t)^2$. Then the operator $T$ defines a class in $KK_0(C(X),\mathbb{C})$ and  we have the following equality in $KK_0(C(X),\mathbb{C})$:$$[T]=\pi_*[\eth_M].$$
\end{proposition}  

\begin{proof}
This proposition is essentially proved in \cite{Haskell-K}. More precisely   \cite{Haskell-K} is devoted to the case of $d_{\max}+d^t_{\min}:L^2\Omega^{\bullet}(\reg(X),\rho)\rightarrow L^2\Omega^{\bullet}(\reg(X),\rho)$, that is the rolled-up operator associated to the maximal de Rham complex. However a careful  analysis of the  proof shows that the same arguments apply verbatim to the Hodge-Dolbeault operator $\overline{\partial}_{0,\min}+\overline{\partial}^t_{0,\max}:L^2\Omega^{0,\bullet}(\reg(X),\rho)\rightarrow L^2\Omega^{0,\bullet}(\reg(X),\rho)$. The fact that $(\overline{\partial}_{0,\min} +  \overline{\partial}_{0,\max}^t ) \circ L^{-1/2}$ defines a class in $KK_0(C(X),\mathbb{C})$ follows  from the next lemma in the case $s=0$.\\
\end{proof}

Let's go back now to the proof of Th. \ref{Kequalities}.
 Recall from \cite{FraBei}  that $
\overline{\partial}_{0,{\rm min}} +  \overline{\partial}_{0,{\rm max}}^t 
: L^2\Omega^{0,\bullet}(\reg(X),h)\rightarrow L^2\Omega^{0,\bullet}(\reg(X),h)$ has discrete spectrum. 
Let $[\overline{\eth}_{{\rm bd}}]\in KK_0 (C(X),\bbC)$ be the class defined by the bounded operator $$ (\overline{\partial}_{0,{\rm min}} +  \overline{\partial}_{0,{\rm max}}^t ) \circ L^{-1/2}$$ as explained in Prop. \ref{othermodel}. We have now the following lemma.
\begin{lemma}
\label{familyhomotopy}
The following equality holds:
$$[\overline{\eth}_{\operatorname{rel}}]=[\overline{\eth}_{{\rm bd}}]\quad\text{in}\quad KK_0(C(X),\bbC)$$
\end{lemma}

\begin{proof}
As $\overline{\partial}_{0,{\rm min}} +  \overline{\partial}_{0,{\rm max}}^t:L^2\Omega^{0,\bullet}(\reg(X),h)\rightarrow L^2\Omega^{0,\bullet}(\reg(X),h)$ is a Fredholm operator we have that $\im(\overline{\partial}_{0,{\rm min}} +  \overline{\partial}_{0,{\rm max}}^t)$ is a closed subspace of $L^2\Omega^{0,\bullet}(\reg(X),h)$. Let us label by $\Pi_{\im}:L^2\Omega^{0,\bullet}(\reg(X),h)\rightarrow \im(\overline{\partial}_{0,{\rm min}} +  \overline{\partial}_{0,{\rm max}}^t)$ the corresponding orthogonal projection. We consider the family of operators $ Q_s:= (\overline{\partial}_{0,{\rm min}} +  \overline{\partial}_{0,{\rm max}}^t ) \circ L_s^{-1/2}$
where 
$L_s = \Pi_{\Ker} + s\Pi_{{\rm Im}} + (\overline{\partial}_{0,{\rm min}} +  \overline{\partial}_{0,{\rm max}}^t)^2$,
$s\in [0,1]$; notice that 
$$L_1^{-1/2}= (\Id + (\overline{\partial}_{0,{\rm min}} +  \overline{\partial}_{0,{\rm max}}^t)^2)^{-1/2}\,.$$
We shall check momentarily that each $Q_s$ defines a $K$-homology cycle for $X$; by homotopy invariance
this will show that $[Q_0]=[Q_1]$ in $KK_0(C(X),\bbC)$, i.e. that $[\overline{\eth}_{{\rm bd}}]=[\overline{\eth}_{\operatorname{rel}}]$
as required.\\
Observe preliminary that $\Pi_{\Ker}$ is obtained by functional calculus 
associated to $(\overline{\partial}_{0,{\rm min}} +  \overline{\partial}_{0,{\rm max}}^t)^2$ 
through a smooth approximation of the characteristic function $\chi_{[-\epsilon,\epsilon]}$
with $\epsilon$ small enough. Similarly $\Pi_{{\rm Im}}$, which  is the identity
minus $\Pi_{\Ker}$, is obtained by functional calculus. 
Hence $L_s^{-1/2}$ is obtained by functional calculus for each $s\in [0,1]$.\\
We first check that for each fixed $s$ the operator $Q_s$ defines a bounded $KK_0(C(X),\bbC)$-cycle.
To this end we need to verify that:

- $Q_s$ is a self-adjoint bounded operator

- $Q_s^2 - \Id$ is a compact operator

- $[M_f, Q_s]$ is compact

\noindent
with $M_f$ the multiplication operator by $f\in S(X)$. \\
The fact that $Q_s$ is bounded and self-adjoint is clear. \\
Let us show that $Q_s^2 - \Id$ is a compact operator. In order to have a lighter notation we set $$P_s:= \Pi_{\Ker} + s\Pi_{{\rm Im}}\;\;\;\text{and}\;\;\;  \Delta_{\overline{\partial},0,{\rm rel}}:=(\overline{\partial}_{0,\min}+\overline{\partial}^t_{0,\max})^2.$$ We have $Q_s^2=\Delta_{\overline{\partial},0,{\rm rel}}\circ (P_s+\Delta_{\overline{\partial},0,{\rm rel}})^{-1}$. Clearly $(P_s+\Delta_{\overline{\partial},0,{\rm rel}})^{-1}$ $:L^2\Omega^{0,\bullet}(\reg(X),h)\rightarrow L^2\Omega^{0,\bullet}(\reg(X),h)$ is a compact operator because $(P_s+\Delta_{\overline{\partial},0,{\rm rel}})^{-1}:L^2\Omega^{0,\bullet}(\reg(X),h)\rightarrow \mathcal{D}(\Delta_{\overline{\partial},0,{\rm rel}})$ is continuous where $\mathcal{D}(\Delta_{\overline{\partial},0,{\rm rel}})$ is endowed with the corresponding graph norm. Consider now any $\omega\in L^2\Omega^{0,\bullet}(\reg(X),h)$ and let $\eta=(P_s+\Delta_{\overline{\partial},0,{\rm rel}})^{-1}\omega$. Let $\eta_1:=\Pi_{\Ker}\eta$ and $\eta_2:=\Pi_{\im}\eta$. Then $\Delta_{\overline{\partial},0,{\rm rel}}\eta=\omega-\eta_1-s\eta_2=\omega-(\Pi_{\Ker}\circ(P_s+\Delta_{\overline{\partial},0,{\rm rel}})^{-1})\omega-(s\Pi_{\im}\circ(P_s+\Delta_{\overline{\partial},0,{\rm rel}})^{-1})\omega$. Therefore $Q_s^2-\Id=-(\Pi_{\Ker}+s\Pi_{\Im})\circ (P_s+\Delta_{\overline{\partial},0,{\rm rel}})^{-1}$ and this allows us to conclude that $Q_s^2-\Id$ is a compact operator.\\
We are left with the task of showing that $[M_f, Q_s]$ is compact. We write $[M_f, Q_s]$ as 
 \begin{align*}&[M_f, (\overline{\partial}_{0,{\rm min}} +  \overline{\partial}_{0,{\rm max}}^t ) \circ L_s^{-1/2}]=\\&
 [M_f, (\overline{\partial}_{0,{\rm min}} +  \overline{\partial}_{0,{\rm max}}^t )] \circ L_s^{-1/2}
 +  (\overline{\partial}_{0,{\rm min}} +  \overline{\partial}_{0,{\rm max}}^t ) \circ [M_f, L_s^{-1/2}]
 \end{align*}
By prop. \ref{bounded} we know that $M_f$ preserves 
$ \mathcal{D}(\overline{\partial}_{0,{\rm min}} +  \overline{\partial}_{0,{\rm max}}^t )$; this implies that the individual
summands on the right hand side are well defined.
 The first summand,  $[M_f, (\overline{\partial}_{0,{\rm min}} +  \overline{\partial}_{0,{\rm max}}^t )] \circ L_s^{-1/2}$, is equal to 
$  {\rm cl}(df)\circ L_s^{-1/2}$ which is certainly compact given that $ {\rm cl}(df)$ is bounded and that $ L_s^{-1/2}$ is compact. We now analyze 
 $$(\overline{\partial}_{0,{\rm min}} +  \overline{\partial}_{0,{\rm max}}^t ) \circ [M_f, L_s^{-1/2}].$$
We can rewrite the latter operator as  
$(\overline{\partial}_{0,{\rm min}} +  \overline{\partial}_{0,{\rm max}}^t ) \circ M_f\circ L_s^{-1/2}-(\overline{\partial}_{0,{\rm min}} +  \overline{\partial}_{0,{\rm max}}^t ) \circ  L_s^{-1/2}\circ M_f$ 
$= {\rm cl}(df)\circ L_s^{-1/2}+M_f\circ(\overline{\partial}_{0,{\rm min}} +  \overline{\partial}_{0,{\rm max}}^t ) \circ L_s^{-1/2}-(\overline{\partial}_{0,{\rm min}} +  \overline{\partial}_{0,{\rm max}}^t ) \circ L_s^{-1/2}\circ M_f$ $={\rm cl}(\overline{\partial}f)\circ L_s^{-1/2}+[M_f, (\overline{\partial}_{0,{\rm min}}+  \overline{\partial}_{0,{\rm max}}^t ) \circ L_s^{-1/2}]$. Clearly ${\rm cl}(df)\circ L_s^{-1/2}$ is compact as $ {\rm cl}(df)$ is bounded and $L_s^{-1/2}$ is compact. Thus we are left with the task
of understanding  the remaining term $[M_f, (\overline{\partial}_{0,{\rm min}}+  \overline{\partial}_{0,{\rm max}}^t ) \circ L_s^{-1/2}]$. By using the analytic functional calculus we can write
$$[M_f, (\overline{\partial}_{0,{\rm min}}+  \overline{\partial}_{0,{\rm max}}^t ) \circ L_s^{-1/2}]=
[M_f, (\overline{\partial}_{0,{\rm min}}+  \overline{\partial}_{0,{\rm max}}^t ) \circ \int_0^\infty \lambda^{-1/2}  
(P_s + \Delta_{\overline{\partial},0,{\rm rel}}+\lambda)^{-1} d\lambda].$$
Reasoning with the definition of the integral we easily justify the equality of the latter term with
$$[M_f, \int_0^\infty \lambda^{-1/2}   (\overline{\partial}_{0,{\rm min}}+  \overline{\partial}_{0,{\rm max}}^t )\circ
(P_s + \Delta_{\overline{\partial},0,{\rm rel}}+\lambda)^{-1} d\lambda]$$
As $[M_f,\cdot]$ is bounded, we finally get 

$$(\overline{\partial}_{0,{\rm min}} +  \overline{\partial}_{0,{\rm max}}^t ) \circ [M_f, L_s^{-1/2}]=\int_0^\infty \lambda^{-1/2}  [M_f,  (\overline{\partial}_{0,{\rm min}}+  \overline{\partial}_{0,{\rm max}}^t )\circ
(P_s + \Delta_{\overline{\partial},0,{\rm rel}}+\lambda)^{-1}] d\lambda.$$
 Now, in order to conclude that  $(\overline{\partial}_{0,{\rm min}} +  \overline{\partial}_{0,{\rm max}}^t ) \circ [M_f, L_s^{-1/2}]$ is compact, it is enough to show that $[M_f,  (\overline{\partial}_{0,{\rm min}}+  \overline{\partial}_{0,{\rm max}}^t )\circ
(P_s + \Delta_{\overline{\partial},0,{\rm rel}}+\lambda)^{-1}]$ is compact. This follows easily by noticing that  $ (\overline{\partial}_{0,{\rm min}}+  \overline{\partial}_{0,{\rm max}}^t )\circ
(P_s + \Delta_{\overline{\partial},0,{\rm rel}}+\lambda)^{-1}$ $=(\overline{\partial}_{0,{\rm min}}+  \overline{\partial}_{0,{\rm max}}^t )\circ
(P_s + \Delta_{\overline{\partial},0,{\rm rel}}+\lambda)^{-1/2}\circ(P_s + \Delta_{\overline{\partial},0,{\rm rel}}+\lambda)^{-1/2}$, that $(\overline{\partial}_{0,{\rm min}}+  \overline{\partial}_{0,{\rm max}}^t )\circ
(P_s + \Delta_{\overline{\partial},0,{\rm rel}}+\lambda)^{-1/2}$ is bounded and that  $
(P_s + \Delta_{\overline{\partial},0,{\rm rel}}+\lambda)^{-1/2}$ is compact.

It remains to show that $Q_s$ is a continuous family of bounded operators with respect to the norm operator topology. Let $\phi$ be a smooth approximation of $\chi_{[-\epsilon,\epsilon]}$. Then we have $Q_s=f_s(\overline{\partial}_{0,{\rm min}} +  \overline{\partial}_{0,{\rm max}}^t)$ with $f_s:=x(\phi+s(1-\phi)+x^2)^{-1/2}$. Since $f_s:\mathbb{R}\times[0,1]\rightarrow \mathbb{R}$ is continuous and bounded the continuity of the family $Q_s$ follows by the well known properties of functional calculus.

\end{proof}

%
The proof of  Theorem \ref{Kequalities} now follows from  Prop. \ref{othermodel} and the Lemma \ref{familyhomotopy}.
\end{proof}

%
%

In the next corollary we point out some geometric consequences of the above theorem.
\begin{corollary}
Let $(X,h)$ be as in Th. \ref{Kequalities}. Then $\pi_*[\overline{\eth}_M]\in KK_0(C(X),\bbC)$ does not depend on the particular resolution $\pi:M\rightarrow X$ that we consider.
\end{corollary}
We conclude this section with the following remark which provides another way to realize the class $\pi_*[\eth_M]$. Consider again a compact and irreducible Hermitian complex space $(X,h)$ and let $\pi:M\rightarrow X$ be a resolution of $X$. Let $D\subset M$ be the divisors with only normal crossings given by $D:=\pi^{-1}(\sing(X))$. Using the map $\pi$, we can induce a Hermitian metric $\gamma$ on $\reg(X)$ by defining $\gamma:=(\pi|_{M\setminus D}^{-1})^*(g|_{M\setminus D})$.  Let $\overline{\pa}^t_0:\Omega_c^{0,\bullet}(M\setminus D)\rightarrow \Omega_c^{0,\bullet}(M\setminus D)$ be the formal adjoint of $\overline{\pa}_0:\Omega_c^{0,\bullet}(M\setminus D)\rightarrow \Omega_c^{0,\bullet}(M\setminus D)$ with respect to $\gamma$. We have the following properties:
\begin{proposition}
\label{gamma}
Let $M$, $X$, $\pi$ and $\gamma$ be as defined above. The operator
\begin{equation}
\label{BP}
\overline{\pa}_0+\overline{\pa}^t_0:L^2\Omega^{0,\bullet}(\reg(X),\gamma)\rightarrow L^2\Omega^{0,\bullet}(\reg(X),\gamma)
\end{equation}
with domain given by $\Omega^{0,\bullet}_c(\reg(X))$ is  essentially self-adjoint. If we label by $\overline{\eth}_{\pi}:L^2\Omega^{0,\bullet}(\reg(X),\gamma)\rightarrow L^2\Omega^{0,\bullet}(\reg(X),\gamma)$ its unique (and hence self-adjoint) extension then $\overline{\eth}_{\pi}$ has discrete spectrum. The operator
\begin{equation}
\label{pushf}
\overline{\eth}_{\pi}:L^2\Omega^{0,\bullet}(\reg(X),\gamma)\rightarrow L^2\Omega^{0,\bullet}(\reg(X),\gamma)
\end{equation}
 defines an   unbounded Fredholm module for $C(X)$ and thus a class $[\overline{\eth}_{\pi}]\in KK_0(C(X),\mathbb{C}).$ Finally we have the following equality in $KK_0(C(X),\bbC)$:
\begin{equation}
\label{pushforward}
\pi_*[\overline{\eth}_M]=[\overline{\eth}_{\pi}].
\end{equation}
\end{proposition}

\begin{proof}
The essential self-adjointness of \eqref{BP} follows by the fact that \eqref{BP} is unitarily equivalent to 
\begin{equation}
\label{BPB}
\overline{\pa}_0+\overline{\pa}^t_0:L^2\Omega^{0,\bullet}(M\setminus D,g|_{M\setminus D})\rightarrow L^2\Omega^{0,\bullet}(M\setminus D,g|_{M\setminus D})
\end{equation} with domain given by $\Omega^{0,\bullet}_c(M\setminus D)$. Now this latter operator is essentially self-adjoint and its unique closed extension coincides with the unique closed extension of \ref{hodgedol}, see \cite{FraBei}, Prop. 3.1. In order to show that \eqref{pushf} defines an even unbounded Fredholm module for $C(X)$ we take as Hilbert space $L^2\Omega^{0,\bullet}(\reg(X),\gamma)$,  the $*$-representation $\upsilon:C(X)\rightarrow \mathcal{B}(L^2\Omega^{0,\bullet}(\reg(X),\gamma))$ is given by pointwise moltiplication, the dense $*$-subalgebra is again $S(X)$, see Def.  \ref{smooth}, and the operator is clearly $\eth_{\pi}$. The proof now follows  by arguing as in the case of Prop. \ref{minclass}. We only need to justify that the domain of \eqref{pushf} is preserved by the action of $C(X)$. First, thanks to Prop. \ref{dense}, it is enough to show that  the domain of \eqref{pushf} is preserved by the action of $S(X)$. Let $\xi$ be defined as $\xi:=\pi^*h$. Then $\xi$ is a smooth semipositive definite Hermitian form on $M$ which is positive definite on $M\setminus D$.  Thus there exists a positive constant $c$ such that $\xi \leq cg$ or, equivalently, $h\leq c\gamma$ on $\reg(X)$.  Labeling by $h^*$ and $\gamma^*$ the Hermitian metrics induced on $T^*\reg(X)$ by $h$ and $\gamma$ respectively, we have $ch^*\geq \gamma^*$, see for instance Prop. 1.8 in \cite{BeiPiazza}.  
Hence, thanks to the latter inequality and Prop. \ref{bbb}, we can deduce that $d(f|_{\reg(X)})\in L^{\infty}\Omega^1(\reg(X),\gamma)$ for any $f\in S(X)$. Finally, by Prop. \ref{bounded}, we can conclude that the domain of \eqref{pushf} is preserved by the action of $S(X)$. For the class on the left hand side of \eqref{pushforward}, a representative is given by  the Fredholm module $(L^2\Omega^{0,\bullet}(M,g),\rho\circ \pi_M^*,\overline{\eth}\circ(1+\overline{\eth}^2)^{-\frac{1}{2}})$ where similarly to the previous case $\rho_M:C(M)\rightarrow \mathcal{B}(L^2\Omega^{0,\bullet}(M,g))$ is the representation given by pointwise multiplication and  $\rho_M\circ \pi^*$ is the representation acting in the following way: $((\rho_M\circ \pi^*)(f))\omega=\rho_M(f\circ \pi)\omega=(f\circ \pi)\omega$ for each $f\in C(X)$ and $\omega\in L^2\Omega^{0,\bullet}(M,g)$. Consider now the map $\pi|_{M\setminus D}:(M\setminus D, g|_{M\setminus D})\rightarrow (\reg(X),\gamma)$. It is clear that it is a holomorphic  isometry of Hermitian manifolds. Hence $\pi^*:L^2\Omega^{0,\bullet}(\reg(X),\gamma)\rightarrow L^2\Omega^{0,\bullet}(M\setminus D, g|_{M\setminus D})$ is a unitary equivalence of Hilbert spaces and as $D$ has measure zero in $M$ with respect to $\dvol_{g}$ we can conclude that $\pi^*:L^2\Omega^{0,\bullet}(\reg(X),\gamma)\rightarrow L^2\Omega^{0,\bullet}(M,g)$ is a unitary equivalence of Hilbert spaces. Now it is immediate to check that $\pi^*$ induces a unitary equivalence between the Fredholm modules $(L^2\Omega^{0,\bullet}(\reg(X),\gamma),\rho_X,\overline{\eth}_{\pi}\circ(1+\overline{\eth}_{\pi}^2)^{-\frac{1}{2}})$ and $(L^2\Omega^{0,\bullet}(M,g),\rho_M\circ \pi^*,\overline{\eth}\circ(1+\overline{\eth}^2)^{-\frac{1}{2}})$ and thereby we can conclude that $\pi_*[\overline{\eth}_M]=[\overline{\eth}_{\pi}]$ as desired. 
\end{proof}

\medskip
\noindent
{\bf More incarnations of $\pi_*[\overline{\eth}_M]$.} We end this section by exploring further realizations of the class $\pi_*[\overline{\eth}_M]$. More precisely, given a compact and irreducible compact complex space $X$ with only isolated singularities and a resolution $\pi:M\rightarrow X$, we will show the existence of some complete Hermitian metrics on $\reg(X)$ whose corresponding Hodge-Dolbeault operator induces a class in $KK_0(C(X),\mathbb{C})$ that equals $\pi_*[ \overline{\eth}_M]$.
We use the construction explained in the statement of Prop. \ref{othermodel}. In particular this makes crucial that $\dim(\sing(X))=0$. As a first example we start by considering a complete K\"ahler manifold $(N,h)$ with finite volume and pinched negative sectional curvatures, that is $-b^2\leq \sec_h\leq -a^2$ for some constants $0<a\leq b$. An important result concerning the geometry of such manifolds is the one proved in \cite{SiuYaucompact} by Siu and Yau. This result  provides the existence of a  compactification of $N$ in terms of a complex projective variety with only isolated singularities. More precisely if $(N,h)$ is a K\"ahler manifold as above then there exists a projective variety $V\subset \mathbb{C}\mathbb{P}^n$ with only isolated singularities such that $\reg(V)$ and $N$ are biholomorphic.

\begin{proposition}
\label{SiuYau}
Let $(N,h)$ be a complete K\"ahler manifold  with finite volume. Assume that the sectional curvatures of $(N,h)$ satisfies  $-b^2\leq \sec_h\leq -a^2$ for some constants $0<a\leq b$. Let $V\subset \mathbb{C}\mathbb{P}^n$ be the Siu--Yau compactification of $N$ and let $\pi:M\rightarrow V$ be a resolution of $V$. Let $\phi:N\rightarrow \reg(V)$ be a biholomorphism and let $\rho$ be the complete K\"ahler metric on $\reg(V)$ induced by $h$ through $\phi$. Then we have the following equality in $KK_0(C(V),\mathbb{C})$:  $$[(\overline{\partial}_{0} +  \overline{\partial}_{0}^* ) \circ L^{-1/2}]=\pi_*[\overline{\eth}_M].$$
\end{proposition}

\begin{proof}
According to \cite{BeiPiazza} Th. 2.12 we know that the hypothesis of Prop. \ref{othermodel} are fulfilled by $V$ and $\rho$. We can therefore conclude that $[(\overline{\partial}_{0} +  \overline{\partial}_{0}^* ) \circ L^{-1/2}]=\pi_*[\overline{\eth}_M]$ in $KK_0(C(V),\mathbb{C})$ as desired.
\end{proof}

The next example we discus is given by the so called Saper-type K\"ahler metrics. These are K\"ahler metrics introduced by Saper in \cite{Saperisolated} in the setting of complex projective varieties with isolated singularities and whose construction was later generalized  by Grant-Melles and Milman in \cite{MellesMilmanPacific} and \cite{MellesMilmanToulouse} to the case of an arbitrary subvariety of a compact K\"ahler manifold. We recall now the definition of Saper-type metric following \cite{MellesMilmanToulouse}.
 Let $V$ be a singular subvariety of a compact complex manifold $M$ and let $\omega$ be the fundamental $(1,1)$-form of a Hermitian metric on $M$. Let $\pi:\wt{M}\rightarrow M$ be a holomorphic map of a compact complex manifold $\wt{M}$ to $M$ whose exceptional set $E$ is a divisor with normal crossing in $\wt{M}$ and such that the restriction $$\pi|_{\wt{M}\setminus E}:\wt{M}\setminus E \longrightarrow M\setminus \sing(V)$$ is a biholomorphism. Let $L_E$ be the line bundle on $\wt{M}$ associated to $E$. Let $s:\wt{M}\rightarrow L_E$ be a global holomorphic section whose  associated divisor $(s)$ equals $E$ (in particular $s$ vanishes exactly on $E$). Let $\gamma$ be any Hermitian metric on $L_E$ such that $\|s\|_{\gamma}$, the norm of $s$ with respect to $\gamma$, satisfies $\|s\|_{\gamma}<1$. A Hermitian metric on $\wt{M}\setminus E$ which is quasi-isometric to a metric with fundamental $(1,1)$-form $$l\pi^*\omega-\frac{\sqrt{-1}}{2\pi}\pa\overline{\pa}\log(\log\|s\|_{\gamma}^2)^2$$ for $l$ a positive integer, will be called a \textbf{Saper-type metric}, distinguished with respect to the map $\pi$. The corresponding metric on $M\setminus {\sing{V}}\cong \wt{M}\setminus E$ and its restriction to $V\setminus \sing{V}$ are also called Saper-type metric.  For existence results we refer to \cite{Saperisolated}, \cite{MellesMilmanPacific} and \cite{MellesMilmanToulouse}.

\begin{proposition}
\label{Saper}
Let $N$ be a compact K\"ahler manifold with K\"ahler form $\omega$ and let $V$ be an analytic subvariety  of  $N$ of complex dimension $v$ such that $\dim(\sing(V))=0$. Let  $\pi:M\rightarrow V$ be a resolution of $V$. Finally let $g_S$ be  a Saper-type metric  on $\reg (V)$ as constructed in  \cite{Saperisolated} or \cite{MellesMilmanToulouse}. Then we have the following equality in $KK_0(C(V),\mathbb{C})$:  $$[(\overline{\partial}_{0} +  \overline{\partial}_{0}^* ) \circ L^{-1/2}]=\pi_*[ \overline{\eth}_M].$$
\end{proposition}
\begin{proof}
According to \cite{Saperisolated} Th.10.2 or \cite{BeiPiazza} Th.2.4 we know that $H^{0,q}_{2,\overline{\pa}}(\reg (V), g_S)\cong H^{0,q}_{\overline{\pa}} (\widetilde{V})$ for each $q=0,...,v$. Hence the desired conclusion follows now by Prop. \ref{othermodel}.
\end{proof}

As a last example of this subsection we discuss the so called Poincar\'e-type K\"ahler metrics. Let $X$ be a compact and irreducible  complex space. Assume that   $\dim(\sing(X))=0$. Assume that there exists a resolution of $X$, $\pi:M\rightarrow X$, carrying a K\"ahler metric $\upsilon$ with fundamental form $\omega$. Let $D$ be the normal crossings divisors given by $D=\pi^{-1}(\sing(X))$. Let $L_D$ be the line bundle on $M$ associated to $D$. Let $s:M\rightarrow L_D$ be a global holomorphic section whose  associated divisor $(s)$ equals $D$. Let $\tau$ be any Hermitian metric on $L_D$ such that $\|s\|_{\tau}<1$. A K\"ahler  metric $g$ on $M\setminus D$ which is quasi-isometric to a K\"ahler metric with fundamental $(1,1)$-form $$b\omega-\frac{\sqrt{-1}}{2\pi}\pa\overline{\pa}\log(\log\|s\|_{\tau}^2)^2$$ for $b$ a positive integer, will be called a \textbf{Poincar\'e-type metric}.

\begin{proposition}
\label{Poinc}
Let $\omega$, $\pi:M\rightarrow X$, $g$ and $D$ be as above. Let us label by $g_P$ the K\"ahler metric on $\reg(X)$ induced by $g$ through $\pi$.  Then we have the following equality in $KK_0(C(V),\mathbb{C})$:  $$[(\overline{\partial}_{0} +  \overline{\partial}_{0}^* ) \circ L^{-1/2}]=\pi_*[ \overline{\eth}_M ].$$
\end{proposition}
\begin{proof}
According to \cite{Esnault} or \cite{Zucker} we know that $H^{0,q}_{2,\overline{\partial}}(\reg(X),g_P)\cong H^{0,q}_{\overline{\partial}}(M)$ for each $q=0,...,m$. Now the statement follow by Prop. \ref{othermodel}.
\end{proof}
We conclude this section by providing a {\bf summary} of the various incarnations of $\pi_*[\eth_M]$. Once more let $(X,h)$ be a compact and irreducible Hermitian space with $\dim(\sing(X))=0$ and let $\pi:M\rightarrow X$ be a resolution. We have seen six different ways to construct  $\pi_*[ \overline{\eth}_M]$. More precisely:
\begin{itemize}
\item $\pi_*[ \overline{\eth}_M ]=[(\overline{\partial}_{0} +  \overline{\partial}_{0}^* ) \circ L^{-1/2}]$, where the latter class can be constructed by using a Saper-type K\"ahler metric, a Poincar\'e metric, the metric $h$ or, when it exists, a complete K\"ahler metric of finite volume and pinched negative sectional curvatures. See Prop. \ref{Saper}, \ref{Poinc}, \ref{othermodel} and Prop. \ref{SiuYau}, respectively.\\

\item $\pi_*[ \overline{\eth}_M ]=[\overline{\eth}_{\mathrm{rel}}]$, where the latter class is constructed by using the metric $h$. See Th. \ref{Kequalities}.\\

\item $\pi_*[ \overline{\eth}_M ]=[\overline{\eth}_{\pi}]$, where the latter class is built by using a Hermitian metric $\gamma$ induced on $\reg(X)$ through $\pi$ by an arbitrarily fixed Hermitian metric $g$ on $M$. Moreover we remark that in this setting we can drop the hypothesis that $\dim(\sing(X))=0$. See Prop. \ref{gamma}.
\end{itemize}


\section{Relationship with the Baum-Fulton-MacPherson class}
Let $V\subset \mathbb{C}\mathbb{P}^n$ be a complex projective variety. We will always assume that $V$ is reduced and irreducible. In this section we investigate the relationship between the Baum-Fulton-MacPherson class  and the analytic $K$-homology classes defined by self-adjoint extensions of $\overline{\pa}_0+\overline{\pa}_0^t:\Omega_c^{0,\bullet}(\reg(V))\rightarrow\Omega_c^{0,\bullet}(\reg(V))$ under the assumption $\dim_{\mathbb{C}}V\leq 2$. The main results show the existence of complex  projective varieties where the Baum-Fulton-MacPherson class cannot be realized as a class induced by a self-adjoint extension of $\overline{\pa}_0+\overline{\pa}_0^t$. We recall that the Baum-Fulton-MacPherson class of a complex projective variety $V$ is the $K$-homology class defined as $\alpha_V([\mathcal{O}_V])\in K_0^{\text{top}}(V)$ and was  introduced by Baum-Fulton-MacPherson in their seminal papers \cite{BFMI} and \cite{BFMII}.  In the latter formula $\alpha_V:K^{\text{hol}}_0(V)\rightarrow K^{\text{top}}_{0}(V)$  denotes the morphism constructed by Baum-Fulton-MacPherson between the Grothedieck group of coherent analytic sheaves on $V$ and the topological $K$-homology of $V$. We refer to \cite{BFMI} and \cite{BFMII} for definitions and properties.

\smallskip
\noindent
 As anticipated in the Introduction we set
$${\rm Td}_K^{{\rm BFM}}(V) := \alpha_V([\mathcal{O}_V])\in K_* (V)\qquad\text{and}\qquad {\rm Td}_*^{{\rm BFM}}(V)=\Ch_*
(\alpha_V([\mathcal{O}_V]))\in H_* (V,\mathbb{Q}) . $$

\begin{theorem}
\label{noL2}
Let $V\subset \mathbb{C}\mathbb{P}^n$ be a complex projective curve such that $\sing(V)\neq \emptyset$. Let $h$ be the Hermitian metric on $\reg(V)$ induced by the Fubini-Study metric of  $\mathbb{C}\mathbb{P}^n$. Then 
for any closed extension $\overline{D}:L^2(\reg(V),h)\rightarrow L^2\Omega^{0,1}(\reg(V),h)$ of $\overline{\partial}:C^{\infty}_c(\reg(V))\rightarrow \Omega^{0,1}_c(\reg(V))$ we have that $$\Ind(\overline{D})\not= \chi(V,\mathcal{O}_V).$$ 
\end{theorem}

\begin{proof}
According to \cite{BPScurves} we know that both $\overline{\pa}_{\max/\min}:L^2(\reg(X),h)\rightarrow L^2\Omega^{0,1}(\reg(X),h)$ are Fredholm operators on their domains endowed with the corresponding graph norm. Let $i: \mathcal{D}(\overline{\pa}_{\min})\hookrightarrow\mathcal{D}(\overline{\pa}_{\max})$ be the natural inclusion of $\mathcal{D}(\overline{\pa}_{\min})$ into $\mathcal{D}(\overline{\pa}_{\max})$. Endowing both $\mathcal{D}(\overline{\pa}_{\min})$ and $\mathcal{D}(\overline{\pa}_{\max})$ with the corresponding graph norms and by the fact that $\overline{\pa}_{\min}=\overline{\pa}_{\max}\circ i$ we get that  $i: \mathcal{D}(\overline{\pa}_{\min})\hookrightarrow\mathcal{D}(\overline{\pa}_{\max})$ is a Fredholm operator whose index is given by $\Ind(i)=-\dim(\mathcal{D}(\overline{\pa}_{\max})/\mathcal{D}(\overline{\pa}_{\min}))$. Altogether this tells us that $$\Ind(\overline{\pa}_{\max})=\Ind(\overline{\pa}_{\min})+\dim(\mathcal{D}(\overline{\pa}_{\max})/\mathcal{D}(\overline{\pa}_{\min})).$$
It is now easy to deduce that any other closed extension $\overline{D}:L^2(\reg(V),h)\rightarrow L^2\Omega^{0,1}(\reg(V),h)$ of $\overline{\partial}:C^{\infty}_c(\reg(V))\rightarrow \Omega^{0,1}_c(\reg(V))$ is Fredholm on its domain endowed with the graph norm. Indeed $\ker(\overline{D})\subset \ker(\overline{\partial}_{\max})$ and so  $\dim(\ker(\overline{D}))<\infty$. On the other hand $\im(\overline{\partial}_{\min})\subset \im(\overline{D})$. Thus we have a natural surjective map $L^2\Omega^{0,1}(\reg(V),h)/\im(\overline{\partial}_{\min})\rightarrow L^2\Omega^{0,1}(\reg(V),h)/\im(\overline{D})$ which tells us that $L^2\Omega^{0,1}(\reg(V),h)/\im(\overline{D})$ is finite dimensional. Moreover, the same argument above,  shows that the index of $\overline{D}$ obeys $\Ind(\overline{D})=\Ind(\overline{\pa}_{\min})+\dim(\mathcal{D}(\overline{D})/\mathcal{D}(\overline{\pa}_{\min}))$. In particular we get that 
\begin{equation}
\label{inequalities}
\Ind(\overline{\pa}_{\min})\leq \Ind(\overline{D})\leq \Ind(\overline{\pa}_{\max}).
\end{equation}
According to \cite{FultonLibro} pag. 360 we have $$\chi(V,\mathcal{O}_V)=\chi(M,\mathcal{O}_M)-\sum_{p\in \sing(V)}\delta_p$$ where $\pi:M\rightarrow V$ is a resolution of $V$ and $\delta_p:=l(\pi_*\mathcal{O}_M/\mathcal{O}_V)_p$, that is the length of the stalk at $p$ of the sheaf $\pi_*\mathcal{O}_M/\mathcal{O}_V$. As $\sing(V)\neq \emptyset$ we have that $\chi(V,\mathcal{O}_V)<\chi(M,\mathcal{O}_M)$.
Let us justify this latter inequality; assume by contradiction  that $\chi(V,\mathcal{O}_V)=\chi(M,\mathcal{O}_M)$, that is $l(\pi_*\mathcal{O}_M/\mathcal{O}_V)_p=0$ for any $p\in \sing(V)$. This means that $(\pi_*\mathcal{O}_M)_p=\mathcal{O}_{V,p}$ for any $p\in X$ and therefore  $\pi_*\mathcal{O}_M=\mathcal{O}_{V}$. As $V$ is a curve we can assume that $\pi:M\rightarrow V$ is a normalization of $V$. Indeed if $\nu:N\rightarrow V$ is a normalization of $V$ then it is in particular a resolution of $V$ and if we consider now any other resolution $\pi:M\rightarrow V$ then $\pi^{-1}\circ \nu:N\rightarrow M$ is a biholomorphism. Thus we can conclude that $\tilde{\mathcal{O}}_V=\pi_*\mathcal{O}_M=\mathcal{O}_{V}$, that is $V$ is normal. But a normal curve is non-singular and this is not consistent with the fact that $\sing(V)\neq \emptyset$. Therefore we necessarily have $\chi(V,\mathcal{O}_V)<\chi(M,\mathcal{O}_M)$. On the other hand Th. 4.1 in \cite{BPScurves} tells us that $\Ind(\overline{\pa}_{\min})=\chi(M,\mathcal{O}_{M})$. Hence by using \eqref{inequalities} we conclude that for any closed extension $\overline{D}:L^2(\reg(V),h)\rightarrow L^2\Omega^{0,1}(\reg(V),h)$ of $\overline{\partial}:C^{\infty}_c(\reg(V))\rightarrow \Omega^{0,1}_c(\reg(V))$ we have that $\Ind(\overline{D})\not= \chi(V,\mathcal{O}_V).$
\end{proof}

We shall now use the above Theorem in order to draw conclusions for K-homology classes. First, however, we state
and prove a result about the existence K-homology classes defined by closed extensions $\overline{D}:L^2(\reg(V),h)\rightarrow L^2\Omega^{0,1}(\reg(V),h)$ of $\overline{\partial}:C^{\infty}_c(\reg(V))\rightarrow \Omega^{0,1}_c(\reg(V))$.

\begin{proposition}
In the setting of Theorem \ref{noL2} the following holds:\\
for any  closed extension $\overline{D}:L^2(\reg(V),h)\rightarrow L^2\Omega^{0,1}(\reg(V),h)$ of $\overline{\partial}:C^{\infty}_c(\reg(V))\rightarrow \Omega^{0,1}_c(\reg(V))$, the operator $\overline{D}+\overline{D}^*:L^2\Omega^{0,\bullet}(\reg(V),h)\rightarrow L^2\Omega^{0,\bullet}(\reg(V),h)$ defines a class $[\overline{D}+\overline{D}^*]\in KK_0(C(V),\mathbb{C})$. 
\end{proposition}

\begin{proof}
Let $\overline{D}:L^2(\reg(V),h)\rightarrow L^2\Omega^{0,1}(\reg(V),h)$ be an arbitrary closed extension of $\overline{\partial}:C^{\infty}_c(\reg(V))\rightarrow \Omega^{0,1}_c(\reg(V))$. First we want to show that $\overline{D}+\overline{D}^*$ induces a class in  $KK_0(C(V),\mathbb{C})$. Let $H=L^2(\reg(V),h)\oplus L^2\Omega^{0,1}(\reg(V),h)$ and let $S_c(V)$ be the dense $*$-subalgebra of $C(V)$. As usual $C(V)$ acts on $H$ by pointwise multiplication. Thanks to Prop. \ref{stable} we know that $\mathcal{D}(\overline{D}+\overline{D}^*)$ is preserved by the action of $S_c(V)$. Furthermore, given $f\in S_c(V)$, we have $[\overline{D}+\overline{D}^*,f]\omega=\overline{\partial}f\wedge \omega-(\overline{\partial}f\wedge)^*\omega$ where $(\overline{\partial}f\wedge)^*$ is the adjoint of the map $\eta\mapsto \overline{\partial}f\wedge\eta$. As $\overline{\pa}f\in \Omega_c^{0,1}(\reg(V))$ we can conclude that $[\overline{D}+\overline{D}^*,f]$ induces a bounded operator on $H$. Finally we are left to show that $\overline{D}+\overline{D}^*$ has compact resolvent. Thanks to \cite{FraBei} Th. 5.1 we know that given an arbitrary closed extension $\overline{D}_{1,0}:L^2\Omega^{1,0}(\reg(V),h)\rightarrow L^2\Omega^{1,1}(\reg(V),h)$ of $\overline{\partial}_{1,0}:\Omega^{1,0}_c(\reg(V))\rightarrow \Omega^{1,1}_c(\reg(V))$, the inclusion $\mathcal{D}(\overline{D}_{1,0})\hookrightarrow L^2\Omega^{1,0}(\reg(V),h)$ is a compact operator, where  $\mathcal{D}(\overline{D}_{1,0})$ is endowed with the corresponding graph norm. Composing with $*$ and $c$ this tells us that for any arbitrary closed extension $\overline{D}^*:L^2\Omega^{0,1}(\reg(V),h)\rightarrow L^2(\reg(V),h)$ of $\overline{\partial}^t:\Omega^{0,1}_c(\reg(V))\rightarrow C^{\infty}_c(\reg(V))$, the inclusion $\mathcal{D}(\overline{D}_{0,1}^*)\hookrightarrow L^2\Omega^{1,0}(\reg(V),h)$ is a compact operator. Consider now $\overline{\partial}_{\max}:L^2(\reg(V),h)\rightarrow L^2\Omega^{0,1}(\reg(V),h)$. By \cite{BPScurves} Th. 3.1 and \cite{OvRu}
Th. 1.2 we know that the inclusion $\mathcal{D}(\overline{\partial}_{\max})\hookrightarrow L^2(\reg(V),h)$ is a compact operator. This implies immediately that given an arbitrary closed extension $\overline{D}:L^2(\reg(V),h)\rightarrow L^2\Omega^{0,1}(\reg(V),h)$ of $\overline{\partial}:C^{\infty}_c(\reg(V))\rightarrow \Omega^{0,1}_c(\reg(V))$, the inclusion $\mathcal{D}(\overline{D})\hookrightarrow L^2(\reg(V),h)$ is compact as well. In this way we are in position to conclude that $\overline{D}+\overline{D}^*$ has compact resolvent. Thus $\overline{D}+\overline{D}^*$ induces a class $[\overline{D}+\overline{D}^*]\in KK_0(C(V),\mathbb{C})$.
\end{proof}

\begin{corollary}
\label{noL2Kcurves}
In the setting of Th. \ref{noL2} there is no  closed extension $\overline{D}:L^2(\reg(V),g)\rightarrow L^2\Omega^{0,1}(M,h)$ of $\overline{\partial}:C^{\infty}_c(\reg(V))\rightarrow \Omega^{0,1}_c(\reg(V))$ such that the class $[\overline{D}+\overline{D}^*]$ coincides with the Baum-Fulton-MacPherson class ${\rm Td}^{{\rm BFM}}_K (V)\in K^{\operatorname{top}}_0 (V)$ through the isomorphism  $K^{\operatorname{top}}_0(X)\cong KK_0(C(X),\mathbb{C})$.
\end{corollary}

\begin{proof}
 Let  $p:V\rightarrow q$ be the map sending $X$ to a point. With a little abuse of notation let us label with $p_*$ both maps $p_*:K^{\text{top}}_0(V)\rightarrow \mathbb{C}$, $p_*:KK_0(C(V),\mathbb{C})\rightarrow \mathbb{C}$ induced by $p:V\rightarrow q$. Then $p_*$ commutes with the identification $K^{\text{top}}_0(V)\cong KK_0(C(V),\mathbb{C})$ and it is well known that $p_* ({\rm Td}^{{\rm BFM}}_K (V))\equiv p_*(\alpha_V[\mathcal{O}_V])=\chi(V,\mathcal{O}_V)$ and $p_*([\overline{D}+\overline{D}^*])=\Ind(\overline{D}+\overline{D}^*)$. Now the Corollary follows immediately by Theorem \ref{noL2} as $\chi(\mathcal{O}_V)<\Ind(\overline{D}+\overline{D}^*)$.
\end{proof}

The next result is concerned with a similar question in the setting of normal complex projective surfaces. In order to state it  we need to introduce some notations. Let $(N,h)$ be a possibly incomplete Hermitian manifold of complex dimension $n$. A closed extension $(L^2\Omega^{p,q}(N,h),\overline{D}_{p,q})$ of $(\Omega_c^{p,q}(N),\overline{\partial}_{0,q})$ is given by a choice of a closed extension $\overline{D}_{p,q}:L^2\Omega^{p,q}(N,h)\rightarrow L^2\Omega^{p,q+1}(N,h)$ of $\overline{\pa}_{p,q}:\Omega_c^{p,q}(N)\rightarrow \Omega_c^{p,q+1}(N)$ for each $q=0,...,n$ such that $\Im(\overline{D}_{p,q})\subset \mathcal{D}(\overline{D}_{p,q+1})$ for any $q=0,...,n$. If the cohomology of the complex $(L^2\Omega^{p,q}(N,h),\overline{D}_{p,q})$ is finite dimensional then we will define $\chi_{2,\overline{D}_{p,q}}(N,h)$ as $\chi_{2,\overline{D}_{p,q}}(N):=\sum_{q}(-1)^q\dim(\ker(\overline{D}_{p,q})/\Im(\overline{D}_{p,q-1}))$. In particular we have $\chi_{2,\overline{\pa}_{p,q,\max/	\min}}(N,h):=\sum_{q}(-1)^q\dim(\ker(\overline{\pa}_{p,q,\max/\min})/\Im(\overline{\pa}_{p,q-1,\max/\min}))$.

\begin{theorem}
\label{noL2surf}
Let $V\subset \mathbb{C}\mathbb{P}^n$ a complex projective surface with $\dim(\sing(V))=0$. Let $\pi:M\rightarrow V$ be a resolution of $V$. Let $h$ be the Hermitian metric on $\reg(V)$ induced by the Fubini-Study metric of  $\mathbb{C}\mathbb{P}^n$. Then any closed extension $(L^2\Omega^{0,q}(\reg(V),h),\overline{D}_{0,q})$ of $(\Omega_c^{0,q}(\reg(V)),\overline{\partial}_{0,q})$ has finite dimensional cohomology and its corresponding Euler characteristic $\chi_{2,\overline{D}_{0,q}}(\reg(V),h)$ satisfies 
$$\chi_{2,\overline{\pa}_{0,q,\max}}(\reg(V),h)\leq \chi_{2,\overline{D}_{0,q}}(\reg(V),h)\leq \chi_{2,\overline{\pa}_{0,q,\min}}(\reg(V),h).$$ If $V$ is normal then we have 
\begin{equation}
\label{inequalitiessurfaces}
\chi_{2,\overline{\pa}_{0,q,\max}}(\reg(V),h)\leq \chi_{2,\overline{D}_{0,q}}(\reg(V),h)\leq \chi_{2, \overline{\pa}_{0,q,\min}}(\reg(V),h)\leq \chi(V,\mathcal{O}_V).
\end{equation}
 Finally if $V$ is normal and $R^1 \pi_* \mathcal{O}_M$ does not vanish then we have 
\begin{equation}
\label{inequalitiessurfaces2}
\chi_{2,\overline{\pa}_{0,q,\max}}(\reg(V),h)\leq \chi_{2,\overline{D}_{0,q}}(\reg(V),h)\leq \chi_{2,\overline{\pa}_{0,q,\min}}(\reg(V),h)<\chi(V,\mathcal{O}_V).
\end{equation}
Therefore, given a complex projective surface which is normal and with non-trivial $R^1 \pi_* \mathcal{O}_M$
and given a closed extension $(L^2\Omega^{0,q}(\reg(V),h),\overline{D}_{0,q})$ of $(\Omega_c^{0,q}(\reg(V)),\overline{\partial}_{0,q})$  we have that $$\chi_{2,\overline{D}_{0,q}}(\reg(V),h)\not= \chi(V,\mathcal{O}_V).$$
\end{theorem}

\begin{proof}
According to \cite{PardonSternJAMS} we know that $H^{0,q}_{2,\overline{\pa}_{\min}}(\reg(V),h)$ is finite dimensional for each $q=0,...,2$. By the fact that $\dim(\sing(X))=0$ we are in position to use  \cite{JRupp} Th. 1.9 to conclude  that $H^{0,q}_{2,\overline{\pa}_{\max}}(\reg(V),h)$ is finite dimensional for each $q=0,...,2$. Moreover, again using the assumption that $\dim(\sing(X))=0$, we can  conclude that  $\overline{\partial}_{\max}:L^2(\reg(V),h)\rightarrow L^2\Omega^{0,1}(\reg(V),h)$ and $\overline{\partial}_{\min}:L^2(\reg(V),h)\rightarrow L^2\Omega^{0,1}(\reg(V),h)$ coincides, see  \cite{GrieserLesch} Th. 1.2 or \cite{PSCrelle}. Consider now any closed extension $(L^2\Omega^{0,q}(\reg(V),h),\overline{D}_{0,q})$ of $(\Omega_c^{0,q}(\reg(V)),\overline{\partial}_{0,q})$. Clearly $\ker(\overline{D}_{0,0})=\mathbb{C}$ and $\Im(\overline{D}_{0,0})=\Im(\overline{\pa}_{\max})=\Im(\overline{\pa}_{\min})$. Therefore $\ker(\overline{D}_{0,1})/\Im(\overline{D}_{0,0})$ is finite dimensional because $\ker(\overline{D}_{0,1})\subset \ker(\overline{\pa}_{0,1,\max})$, $\ker(\overline{\pa}_{0,1,\max})/\Im(\overline{\pa}_{\max})$ is finite dimensional and, as previously remarked, we have $\Im(\overline{D}_{0,0})=\Im(\overline{\pa}_{\max})=\Im(\overline{\pa}_{\min})$. Concerning $L^2\Omega^{0,2}(\reg(V),h)/(\Im(\overline{D}_{0,1}))$ we can conclude similarly that it is finite dimensional because $L^2\Omega^{0,2}(\reg(V),h)/(\Im(\overline{\pa}_{0,1,\min}))$ is finite dimensional and $\Im(\overline{\pa}_{0,1,\min})\subset \Im(\overline{D}_{0,1})$. Hence we can conclude that $(L^2\Omega^{0,q}(\reg(V),h),\overline{D}_{0,q})$ has finite dimensional cohomology and its Euler characteristic satisfies $$\chi_{2,\overline{D}_{0,q}}(\reg(V),h)= 1+\dim(L^2\Omega^{0,2}(\reg(V),h)/(\Im(\overline{D}_{0,1})))-\dim(\ker(\overline{D}_{0,1})/\Im(\overline{\partial}_{\min})).$$ It therefore clear that $$\chi_{2,\overline{\pa}_{0,q,\max}}(\reg(V),h)\leq \chi_{2,\overline{D}}(\reg(V),h)\leq \chi_{2,\overline{\pa}_{0,q,\min}}(\reg(V),h)$$  because we have $\dim(L^2\Omega^{0,2}(\reg(V),h)/(\Im(\overline{\pa}_{0,1,\min})))\geq\dim(L^2\Omega^{0,2}(\reg(V),h)/(\Im(\overline{D}_{0,1})))\geq$ $\dim(L^2\Omega^{0,2}(\reg(V),h)/(\Im(\overline{\pa}_{0,1,\max})))$ and $\dim(\ker(\overline{\pa}_{0,1,\min})/\Im(\overline{\partial}_{\min}))\leq \dim(\ker(\overline{D}_{0,1})/\Im(\overline{\partial}_{\min}))$ $\leq \dim(\ker(\overline{\pa}_{0,1,\max})/\Im(\overline{\partial}_{\min}))$.
Concerning \eqref{inequalitiessurfaces} and \eqref{inequalitiessurfaces2} we argue as follows. Thanks to \cite{FultonLibro} pag. 361 we know that,  given any normal surface $V\subset \mathbb{C}\mathbb{P}^n$ its structure sheaf satisfies $\chi(V,\mathcal{O}_V)=\chi(M,\mathcal{O}_M)+\sum_{p}n_p$ where the sum is taken over the points $p\in \sing(V)$ and $n_p:= l(R^1 \pi_* \mathcal{O}_M)_p$, that is the length of the stalk of the sheaf $R^1 \pi_* \mathcal{O}_M$ at $p$. Clearly $\sum_pn_p\geq 0$ and if $R^1\pi_*\mathcal{O}_M$ is non-trivial then $\sum_pn_p>0$. Now \eqref{inequalitiessurfaces} and \eqref{inequalitiessurfaces2} follow immediately because thanks to \cite{Pardon}, \cite{PardonSternJAMS} and \cite{JRupp} we know that $\chi_{2,\overline{\pa}_{0,q,\min}}(\reg(V),h)=\chi(M,\mathcal{O}_M)$.
\end{proof}

We point out that the inequality $\chi_{2,\overline{\pa}_{0,q,\max}}(\reg(V),h)\leq  \chi_{2,\overline{\pa}_{0,q,\min}}(\reg(V),h)$ has been already proved in \cite{Pardon}.

We now proceed to use this result in order to compare K-homology classes, as it was done for algebraic curves.
First we establish the existence of analytic K-homology classes corresponding to closed extensions 
$(L^2\Omega^{0,q}(N,h),\overline{D}_{0,q})$ of $(\Omega_c^{0,q}(N),\overline{\partial}_{0,q})$.

\begin{proposition}
\label{noKL2-prop}
Let $V\subset \mathbb{C}\mathbb{P}^n$ be a complex projective surface with $\dim(\sing(V))=0$. Let $h$ be the metric on $\reg(V)$ induced by the Fubini-Study metric of $\mathbb{C}\mathbb{P}^n$. Given an arbitrary closed extension  $(L^2\Omega^{0,q}(N,h),\overline{D}_{0,q})$ of $(\Omega_c^{0,q}(N),\overline{\partial}_{0,q})$ 
the corresponding  rolled-up operator $\overline{D}_0+\overline{D}_0^*:L^2\Omega^{0,\bullet}(\reg(V),h)\rightarrow L^2\Omega^{0,\bullet}(\reg(V),h)$ defines a class $[\overline{D}_0+\overline{D}^*_0]\in KK_0(C(V),\mathbb{C})$.
\end{proposition}

\begin{proof}
Let $(L^2\Omega^{0,q}(\reg(V),h),\overline{D}_{0,q})$ be any closed extension of $(\Omega_c^{0,q}(\reg(V)),\overline{\partial}_{0,q})$ and let  $\overline{D}_0+\overline{D}_0^*:L^2\Omega^{0,\bullet}(\reg(V),h)\rightarrow L^2\Omega^{0,\bullet}(\reg(V),h)$ be the corresponding rolled-up operator.  Let $H= L^2\Omega^{0,\bullet}(\reg(V),h)$ and let us fix  $S_c(V)$ as a dense $*$-subalgebra of $C(V)$. As in the previous cases $C(V)$ acts on $H$ by pointwise multiplication. Thanks to Prop. \ref{stable} we know that $\mathcal{D}(\overline{D}_0+\overline{D}_0^*)$ is preserved by the action of $S_c(V)$. Furthermore, given $f\in S_c(V)$, we have $[\overline{D}_0+\overline{D}_0^*,f]\omega=\overline{\partial}f\wedge \omega-(\overline{\partial}f\wedge)^*\omega$ where $(\overline{\partial}f\wedge)^*$ is the adjoint of the map $\eta\mapsto \overline{\partial}f\wedge\eta$. As $\overline{\pa}f\in \Omega_c^{0,1}(\reg(V))$ we can conclude that $[\overline{D}_0+\overline{D}_0^*,f]$ induces a bounded operator on $H$. Thus we are left with the task  to prove that $\overline{D}_0+\overline{D}_0^*$ has compact resolvent. This latter assertion is well known to be equivalent to saying that $\overline{D}_0+\overline{D}_0^*$ has entirely discrete  spectrum. Moreover the spectrum of  $\overline{D}_0+\overline{D}_0^*$  is discrete if and only if the spectrum of its square $(\overline{D}_0+\overline{D}_0^*)^2:L^2\Omega^{0,\bullet}(\reg(V),h)\rightarrow L^2\Omega^{0,\bullet}(\reg(V),h)$ is discrete. On the other hand $(\overline{D}_0+\overline{D}_0^*)^2$ decomposes as the direct sum of three self-adjoint operators: $\overline{D}^*_{0,0}\circ \overline{D}_{0,0}:L^2(\reg(V),h)\rightarrow L^2(\reg(V),h)$, $\overline{D}^*_{0,1}\circ \overline{D}_{0,1}+\overline{D}_{0,0}\circ \overline{D}_{0,0}^*:L^2\Omega^{0,1}(\reg(V),h)\rightarrow L^2\Omega^{0,1}(\reg(V),h)$ and $\overline{D}_{0,1}\circ \overline{D}_{0,1}^*:L^2\Omega^{0,2}(\reg(V),h)\rightarrow L^2\Omega^{0,2}(\reg(V),h)$. Therefore, in order to show that $\overline{D}_0+\overline{D}_0^*$ has discrete spectrum, it suffices to prove that the three operators above have discrete spectrum. By \cite{GrieserLesch} Th. 1.2 we know that $\overline{D}_{0,0}^*\circ \overline{D}_{0,0}:L^2(\reg(V),h)\rightarrow L^2(\reg(V),h)$ has discrete spectrum. According to \cite{FraBei} Th. 5.1 we know that $\overline{D}_{0,1}\circ \overline{D}_{0,1}^*:L^2\Omega^{0,2}(\reg(V),h)\rightarrow L^2\Omega^{0,2}(\reg(V),h)$ has discrete spectrum. Finally, as we know that $\ker(\overline{D}_{0,1})/\Im(\overline{D}_{0,0})\cong \ker(\overline{D}^*_{0,1}\circ \overline{D}_{0,1}+\overline{D}_{0,0}\circ \overline{D}_{0,0}^*)$ is finite dimensional and that both $\overline{D}_{0,1}\circ \overline{D}_{0,1}^*$ and $\overline{D}_{0,0}^*\circ \overline{D}_{0,0}$ have discrete spectrum, we can use \cite{FraBei} Cor. 2.1 in order to conclude that $\overline{D}^*_{0,1}\circ \overline{D}_{0,1}+\overline{D}_{0,0}\circ \overline{D}_{0,0}^*:L^2\Omega^{0,1}(\reg(V),h)\rightarrow L^2\Omega^{0,1}(\reg(V),h)$ has discrete spectrum too. This establishes the first point of this corollary. \end{proof}

\begin{corollary}
\label{noKL2}
Let $V\subset \mathbb{C}\mathbb{P}^n$ be a complex projective surface with $\dim(\sing(V))=0$. Let $h$ be the metric on $\reg(V)$ induced by the Fubini-Study metric of $\mathbb{C}\mathbb{P}^n$. We assume that $V$ is normal and 
that $R^1\pi_*\mathcal{O}_M$ is non-trivial.  Then there is no closed extension $(L^2\Omega^{0,q}(V,h),\overline{D}_{0,q})$ of $(\Omega_c^{0,q}(V),\overline{\partial}_{0,q})$  such that $[\overline{D}_0+\overline{D}_0^*]$ coincides with the Baum-Fulton-MacPherson class ${\rm Td}^{{\rm BFM}}_K (V)\in K^{\operatorname{top}}_0 (V)$ through the isomorphism  $K^{\operatorname{top}}_0(X)\cong KK_0(C(X),\mathbb{C})$. 

\end{corollary}

\begin{proof}
Thanks to Theorem \ref{noL2surf} we know that $\Ind(\overline{D}_0+\overline{D}_0^*)=\chi_{2,\overline{D}_{0,q}}(\reg(V),h)<\chi(V,\mathcal{O}_V)$. Now the assertion follows by arguing as in the proof of Cor. \ref{noL2Kcurves}.
\end{proof}

\smallskip
\noindent
{\bf Remark.} The above results show that there are no  closed extension $(L^2\Omega^{0,q}(V,h),\overline{D}_{0,q})$ of $(\Omega_c^{0,q}(V),\overline{\partial}_{0,q})$ with the property that the associated rolled-up operator defines a K-homology class
realising analytically the Baum-Fulton-MacPherson. Still, there might exist a {\it different} Hilbert complex with such a property.
Recently John Lott has constructed such a Hilbert complex. See \cite{Lott}.

\smallskip
\noindent
We end this section with the following proposition that we believe to have an independent interest.

\begin{proposition}
Let $V\subset \mathbb{C}\mathbb{P}^n$ a complex projective surface with $\dim(\sing(V))=0$.  Let $h$ be the Hermitian metric on $\reg(V)$ induced by the Fubini-Study metric of  $\mathbb{C}\mathbb{P}^n$. Then the quotient $\mathcal{D}(\overline{\pa}_{0,1,\max})/\mathcal{D}(\overline{\pa}_{0,1,\min})$ is a finite dimensional vector space and we have the following formula $$\dim(\mathcal{D}(\overline{\pa}_{0,1,\max})/\mathcal{D}(\overline{\pa}_{0,1,\min}))=\chi(M,\mathcal{O}_M)-\chi(M,\mathcal{O}(Z-|Z|))$$ where $\pi:M\rightarrow V$ is a resolution of $M$, $Z$ is the unreduced exceptional set of $\pi$ and $|Z|$ is the reduced exceptional set of $\pi$.
\end{proposition}

\begin{proof}
We already know that $\ker(\overline{\pa}_{0,q,\max})/\Im(\overline{\pa}_{0,q-1,\min})$ is finite dimensional for each $q=0,...,2$. In fact, when $q=0$, this is just $\ker(\overline{\pa}_{0,q,\max})$ and when $q=2$ it becomes $L^2\Omega^{0,2}(\reg(V),h)/\Im(\overline{\pa}_{0,1,\min})$ $=$ $H^{0,2}_{2,\overline{\pa}_{\min}}(\reg(V),h)$. Finally when $q=1$ we have $\ker(\overline{\pa}_{0,1,\max})/\Im(\overline{\pa}_{\min})=$ $\ker(\overline{\pa}_{0,1,\max})/\Im(\overline{\pa}_{\max})$ $=H^{0,2}_{2,\overline{\pa}_{\max}}(\reg(V),h)$ as we have already recalled above that $\overline{\pa}_{\min}=\overline{\pa}_{\max}$. Hence we are in position to apply \cite{FBJTA} Th. 1.4 and Cor. 1.1 and this tells us that $\mathcal{D}(\overline{\pa}_{0,1,\max})/\mathcal{D}(\overline{\pa}_{0,1,\min})$ is finite dimensional and that 
\begin{align}
& \nonumber \dim(\mathcal{D}(\overline{\pa}_{0,1,\max})/\mathcal{D}(\overline{\pa}_{0,1,\min}))=\\
& \nonumber  \dim(H^{0,1}_{2,\overline{\pa}_{\max}}(\reg(V),h))-\dim(H^{0,2}_{2,\overline{\pa}_{\max}}(\reg(V),h))+\dim(H^{0,2}_{2,\overline{\pa}_{\min}}(\reg(V),h))-\dim(H^{0,1}_{2,\overline{\pa}_{\min}}(\reg(V),h))\\
& \nonumber =\sum_{q=0}^2\dim(H^{0,2}_{2,\overline{\pa}_{\min}}(\reg(V),h))-\sum_{q=0}^2\dim(H^{0,2}_{2,\overline{\pa}_{\max}}(\reg(V),h))= \chi(M,\mathcal{O}_M)-\chi(M,\mathcal{O}(Z-|Z|))
\end{align}
where the last equality follows by the results proved in \cite{PardonSternJAMS} and \cite{JRupp}. The proof is thus complete.
\end{proof}

\noindent
There are many examples of normal projective surfaces with non-rational singularities.  For instance any surface  $S\subset \mathbb{C}\mathbb{P}^3$ which is a projective cone over a plane smooth curve having degree  bigger than 2 is a normal projective surface with non-rational singularities. A simple example is provided by the surface $S\subset \mathbb{C}\mathbb{P}^3$ defined by  $X^3+Y^3+Z^3=0$. More generally   Artin's criterion, see \cite{compactcomplexsurfaces} page 94, can be used to construct examples of normal projective surfaces with non-rational singularities by contracting exceptional curves. Finally we mention that other interesting examples of normal projective surfaces with non-rational singularities are given  in \cite{Cheltsov} Section 3.

\section{Rational singularities.}
 We begin this section by
 recalling that  in the context of complex spaces Levy has generalized the results of Baum-Fulton-MacPherson,
defining in particular a homomorphism $\alpha_X:K^{\operatorname{hol}}_0(X)\rightarrow K^{\operatorname{top}}_0(X)$,
with $K^{\operatorname{hol}}_0(X)$ equal to the $K$-homology group of coherent analytic sheaves on the complex 
space $X$. See \cite{Roni}.

 In this section we are interested in complex spaces with rational singularities.
 Recall that a complex space $X$ is said to have rational singularities if $X$ is normal and there exists a resolution $\pi:M\rightarrow X$ such that $R^k\pi_*\mathcal{O}_M=0$ for $k>0$. 
In the setting of complex projective varieties well known examples are provided by  log-terminal singularities, canonical singularities and toric singularities. In the framework of complex surfaces another well known class is provided by  Du Val singularities.  As a reference for this topic we recommend \cite{Artin}, \cite{JKol}, \cite{KoMo}, \cite{Laufer} and \cite{Miles}. Let $X$ be a compact and irreducible complex space with only rational singularities. Let $\pi:M\rightarrow X$ be a resolution of $X$. Let $g$ be any Hermitian metric on $M$ and let $\gamma$ be the Hermitian metric on $\reg(X)$ induced by $g$ through $\pi$. Consider the complex of preasheves  given by the assignment $U\mapsto \mathcal{D}(\overline{\partial}_{0,q,\max})\subset L^2\Omega^{0,q}(\reg(U),\gamma|_{\reg(U)})$ and let us denote by $(\mathcal{L}^{0,q}, \overline{\partial}_{0,q})$ the corresponding complex of sheaves arising by sheafification. We a little abuse of notation we have labeled by $\overline{\partial}_{0,q}:\mathcal{L}^{0,q}\rightarrow \mathcal{L}^{0,q+1}$ the morphism of sheaves induced by $\overline{\partial}_{0,q,\max}$. We have the following property:
\begin{proposition}
\label{ratio}
Let $X$ be a compact and irreducible  complex space. Then the complex of fine sheaves $(\mathcal{L}^{0,q}, \overline{\partial}_{0,q})$ is a  resolution of $\mathcal{O}_X$ if and only if $X$ has rational singularities.
\end{proposition}
\begin{proof}
Let $D\subset M$ be the normal crossing divisor given by $D=\pi^{-1}(\sing(X))$. Consider the complex of preasheves  given by the assignment $U\mapsto \mathcal{D}(\overline{\partial}_{0,q,\max})\subset L^2\Omega^{0,q}(U\setminus (U\cap D),g|_{U\setminus (U\cap D)})$ and let us denote by $(\mathcal{C}^{0,q}_{D}, \overline{\partial}_{0,q})$ the corresponding complex of sheaves arising by sheafification. Besides $(\mathcal{C}^{0,q}_{D}, \overline{\partial}_{0,q})$ let us consider also the complex of sheaves $(\mathcal{C}^{0,q}, \overline{\partial}_{0,q})$ defined as sheafification of the complex of preasheves  given by $U\mapsto \mathcal{D}(\overline{\partial}_{0,q,\max})\subset L^2\Omega^{0,q}(U,g|_U)$. Arguing as in \cite{PardonSternJAMS} Prop. 1.12 and 1.17 we can show that we have an equality of complexes of sheaves $(\mathcal{C}^{0,q}, \overline{\partial}_{0,q})$ $=(\mathcal{C}^{0,q}_D, \overline{\partial}_{0,q})$. As $(\mathcal{C}^{0,q}, \overline{\partial}_{0,q})$ is a fine resolution of $\mathcal{O}_M$ we know that  $(\mathcal{C}^{0,q}_D, \overline{\partial}_{0,q})$ is a fine resolution of $\mathcal{O}_M$ as well. Assume now that $X$ has rational singularities. It is clear that $\ker(\mathcal{L}^{0,0}\stackrel{\overline{\partial}}{\rightarrow}\mathcal{L}^{0,1})=\pi_*\ker(\mathcal{C}_D^{0,0}\stackrel{\overline{\partial}}{\rightarrow}\mathcal{C}_D^{0,1})$. On the other hand $\ker(\mathcal{C}_D^{0,0}\stackrel{\overline{\partial}}{\rightarrow}\mathcal{C}_D^{0,1})$ $=\ker(\mathcal{C}^{0,0}\stackrel{\overline{\partial}}{\rightarrow}\mathcal{C}^{0,1})=\mathcal{O}_M$. Thus we showed that 
$\ker(\mathcal{L}^{0,0}\stackrel{\overline{\partial}}{\rightarrow}\mathcal{L}^{0,1})=\pi_*\mathcal{O}_M$ and since $X$ is normal we have $\mathcal{O}_X=\pi_*\mathcal{O}_M=\ker(\mathcal{L}^{0,0}\stackrel{\overline{\partial}}{\rightarrow}\mathcal{L}^{0,1})$.
Finally, as $X$ has rational singularities and $\mathcal{L}^{0,q}=\pi_*\mathcal{C}_D^{0,q}$, we can conclude that $(\mathcal{L}^{0,q},\overline{\partial}_{0,q})$ is an exact complex of  sheaves and thus a  resolution of $\mathcal{O}_X$. Assume now that the complex of sheaves $(\mathcal{L}^{0,q}, \overline{\partial}_{0,q})$ is a  resolution of $\mathcal{O}_X$. As observed above for any compact and irreducible Hermitian complex space we have $\ker(\mathcal{L}^{0,0}\stackrel{\overline{\partial}}{\rightarrow}\mathcal{L}^{0,1})=\pi_*\mathcal{O}_M$. On the other hand it is clear that $\pi_*\mathcal{O}_M=\tilde{\mathcal{O}}_X$. As we assumed that $(\mathcal{L}^{0,q}, \overline{\partial}_{0,q})$ is a  resolution of $\mathcal{O}_X$ we are led to conclude that $\mathcal{O}_X=\tilde{\mathcal{O}}_X$, that is $X$ is normal. Finally by the fact that $\mathcal{L}^{0,q}=\pi_*\mathcal{C}_D^{0,q}$ and that by assumption $(\mathcal{L}^{0,q}, \overline{\partial}_{0,q})$ is a  resolution of $\mathcal{O}_X$  we can conclude that $R^k\pi_*\mathcal{O}_M=0$ for each $k>0$. In conclusion $X$ has rational singularities as desired.
\end{proof}

We recall also that $X$ has rational singularities if and only if the structure sheaf $\mathcal{O}_X$ can be resolved by using a complex of sheaf built from the minimal $L^2$-$\overline{\partial}$ complex, see \cite{RUIMRN}.
It is therefore natural to expect that in this setting the class $\alpha_X [\mathcal{O}_X]$, which is the
generalization given by Levy of the Baum-Fulton-MacPherson class,
coincides with the pushforward  of the analytic Todd class of $M$. We now proceed to establish this result.
\vspace{0.2cm} 

Let $(X,h)$ be a  compact and irreducible Hermitian complex space with only rational singularities. Since, by assumption,
$X$ is normal we know that $\mathcal{O}_X=\pi_*\mathcal{O}_M$. Therefore in $K^{\operatorname{hol}}_0(X)$
we have 
\begin{equation}
\label{chain}
[\mathcal{O}_X]=[\pi_*\mathcal{O}_M]=\pi_{!}([\mathcal{O}_M]).
\end{equation}
Now using the map $\alpha_X:K^{\operatorname{hol}}_0(X)\rightarrow K^{\operatorname{top}}_0(X)$, defined by  Levy in 
\cite{Roni},
and the fact that $K^{\operatorname{top}}_0(X)\cong KK_0(C(X),\mathbb{C})$ we can conclude that \begin{equation}
\label{chain2}
\alpha_X([\mathcal{O}_X])=\alpha_X(\pi_{!}([\mathcal{O}_M]))=\pi_*(\alpha_M([\mathcal{O}_M]))=\pi_* [\overline{\eth}_M]
\end{equation}
In the projective case this means that
\begin{equation}
\label{chain2-bis}
{\rm Td}^{{\rm BFM}}_K (X)=\pi_* [\overline{\eth}_M]
\end{equation}
Consequently, by Theorem \ref{Kequalities}, we obtain that under the additional assumption that $\dim (\sing (X))=0$
we have that $$\alpha_X([\mathcal{O}_X])=[\overline{\eth}_{\operatorname{rel}}]$$
and if  $X$ is a projective variety, this means that
\begin{equation}\label{chain2-ter}
{\rm Td}^{{\rm BFM}}_K (X)=[\overline{\eth}_{\operatorname{rel}}]\,.
\end{equation}
In the  equality \eqref{chain2} we used  that $\alpha_M([\mathcal{O}_M])=[\overline{\eth}_M]$ and we now explain why this equality holds. By Lemma 3.4.(c) in Levy's article \cite{Roni} we know that $\alpha_M([\mathcal{O}_M])$ is the Poincar\'e dual 
of the class in $K^0 (M)$ corresponding to the locally free sheaf given by $\mathcal{O}_M$; the latter is obviously the 
product line bundle $M\times \mathbb{C}$ over $M$. The Poincar\'e dual of this element is obtained by applying the Thom
isomoprhism $K^0 (M) \to K^0 (\Lambda^{1,0} M)$ followed by the quantization isomorphism $ K^0 (\Lambda^{1,0} M)\to K_0 (M)$. But the Thom isomorphism applied to the trivial line bundle is equal to the class
defined by the symbol of $\overline{\pa}+\overline{\pa}^t$, see \cite{LawMich}  Theorem C8 p. 387,  and the associated K-homology class is precisely 
$[\overline{\eth}_M]$, as required.
 Summarizing we have:
 
 \begin{proposition}\label{prop-equality-rational}
If $(X,h)$ is  a  compact and irreducible Hermitian complex space, then the analogue of the Baum-Fulton-MacPherson class constructed by Levy, $\alpha_X([\mathcal{O}_X])\in K^{\operatorname{top}}_0 (X)$,
coincides with $\pi_* [\overline{\eth}_M]$. In addition, requiring that $\dim(\sing(X))=0$, the class $\alpha_X([\mathcal{O}_X])\in K^{\operatorname{top}}_0 (X)$ coincides also with  $[\overline{\eth}_{\operatorname{rel}}]$ through the isomorphism 
$K^{\operatorname{top}}_0(X)\cong KK_0(C(X),\mathbb{C})$. Consequently,
$\Ch_* [\overline{\eth}_{\operatorname{rel}}]=\Ch_* (\alpha_X([\mathcal{O}_X]))$ in $H_* (X,\mathbb{Q})$. In particular,
if $X$ is projective then 
\begin{equation}\label{eq:equality-rational}
\Ch_* [\overline{\eth}_{\operatorname{rel}}]={\rm Td}^{{\rm BFM}}_* (X) \quad\text{in}\quad H_* (X,\mathbb{Q}).
\end{equation}
\end{proposition}


%
%
%
\section{Invariance by birational equivalence}
 Let $V$ be a complex projective variety. Let $\Gamma:=\pi_1(V)$ be its fundamental group, let $B\Gamma$ be the classifying space of $\Gamma$ and let $r:V\rightarrow B\Gamma$ be a classifying map for the universal covering of $V$,  $b:\widetilde{V}\rightarrow V$. 
 Assume first that $V$ is smooth and consider $[\overline{\eth}_{V}]\in K_0 (V)$. Recall that  $[\overline{\eth}_{V}]={\rm Td}_K (V):= \alpha_V [\mathcal{O}_V]$ in $K_0 (V)$. It is proved in 
  \cite{WBlock} that the class $r_*  ( {\rm Td}_K (V))$  in $K_0(B\Gamma)$
   is a birational invariant of $V$. This means the following. Let $\psi: W\dashrightarrow V$ be a birational equivalence;
   it is well known, see \cite{Griffiths-Harris}, that $\psi$ induces
 an isomorphism between the fundamental groups of $W$ and $V$ and thus a homotopy equivalence 
   between the respective classifying spaces. Put it differently, we can choose 
   $B\Gamma$ as a classifying space for the universal covering of $W$, $\widetilde{W}\to W$. 
   If now $s:W\to B\Gamma$ is a classifying map
   associated to $\widetilde{W}\to W$, then the birational invariance we have alluded to means that 
    $$ s_* ({\rm Td}_K (W)) = r_* ({\rm Td}_K (V))\quad\text{in}\quad K_0(B\Gamma)\,.$$
    We can rewrite this as
    \begin{equation}\label{BLWE}
    s_* [ \overline{\eth}_{W}] = r_* ( [\overline{\eth}_{V} ] )\quad\text{in}\quad K_0(B\Gamma)\,.
    \end{equation}
It is then clear  that the higher Todd genera of $V$, defined as
  $$\{\langle \alpha, r_* {\rm Td}_* (V) \rangle\,,\quad \alpha\in H_* (B\Gamma)\}\,,$$
  are {\it birational invariants} of $V$.
  
  \medskip
   In this section we want to investigate the analogue of these properties in the singular case.
  The first important remark we have to make is that, unlike  in the smooth case, in the singular case the fundamental group is not a birational invariant. We can consider for instance a smooth plane curve $C$ of positive genus. Its projective cone is simply connected but this is not true for its resolution  which is a $\mathbb{P}^1$-bundle over $C$.
  There are, however, interesting special cases in which it is. One of these is provided by complex projective surfaces with only rational singularities, see \cite{Briesk}. Another important class is provided by  projective varieties with log-terminal singularities, see \cite{Takayama}. We shall include these particular cases in the following general situation:
  
  \medskip\noindent
    $V$ and $W$ will be two complex projective varieties 
  with $\dim(\sing(V))=\dim(\sing(W))=0$ and with $\psi:W\dashrightarrow V$  a birational equivalence between them;
  we will  assume that there exist resolutions $\pi:M\rightarrow V$ and $\rho:N\rightarrow W$ such that both maps $\pi_*:\pi_1(M)\rightarrow \pi_1(V)$ and $\rho_*:\pi_1(N)\rightarrow \pi_1(W)$ are isomorphisms.
  Notice that, consequently, $V$ and $W$ have isomorphic fundamental groups. Indeed we know that both $\pi_*:\pi_1(M)\rightarrow \pi_1(V)$ and $\rho_*:\pi_1(N)\rightarrow \pi_1(W)$ are isomorphisms. Moreover $\psi$, $\pi$ and $\rho$ induce a birational map $\lambda: N\dashrightarrow M$ which in turn induces an isomorphism between $\pi_1(M)$ and $\pi_1(N)$.  Summarizing: $\pi_1(W)\cong \pi_1(N)\cong \pi_1(M)\cong \pi_1(V)$. 
%
We can thus identify  the classifying spaces for the universal coverings of $W$ and $V$ with a common space $B\Gamma$.
We have now all the ingredients for the main result of this section:

\begin{proposition}\label{sTability}
Let $\psi: W\dashrightarrow V$ be a  birational equivalence between  complex projective varieties
with $\dim(\sing(V))=\dim(\sing(W))=0$. Assume that there exist resolutions $\pi:M\rightarrow V$ and $\rho:N\rightarrow W$ such that both maps $\pi_*:\pi_1(M)\rightarrow \pi_1(V)$ and $\rho_*:\pi_1(N)\rightarrow \pi_1(W)$ are isomorphisms.   Let $s: W\rightarrow B\Gamma$ and $r:V\rightarrow B\Gamma$ be  classifying maps for the universal coverings 
    $\widetilde{W}\rightarrow W$ and $\widetilde{V}\rightarrow V$. Then, with the above notations,
 $$s_*[\overline{\eth}^W_{\mathrm{rel}}]=r_*[\overline{\eth}^V_{\mathrm{rel}}]\quad\text{in}\quad K_0(B\Gamma)\,.$$
\end{proposition}

\noindent
In order to prove this proposition we need the following lemma.

\begin{lemma}
\label{sTability2}
Let $V$ be a complex projective variety
with $\dim(\sing(V))=0$. Assume that there exists a resolution $\pi:M\rightarrow V$ such that $\pi_*:\pi_1(M)\rightarrow \pi_1(V)$ is an isomorphism.  
Set $\Gamma:= \pi_1 (V)$ and let  $\ell: M\rightarrow B\Gamma$ and $r: V\rightarrow B\Gamma$ be  classifying maps for 
$a:\widetilde{M}\rightarrow M$ and $b:\widetilde{V}\rightarrow V$, the universal coverings of $M$ and $V$ respectively. 
Then the following equality holds: 
$$\ell_*[\overline{\eth}_{M}]=r_*[\overline{\eth}^V_{\mathrm{rel}}] \quad\text{in}   \quad K_0(B\Gamma)\,.$$

\end{lemma}

\begin{proof}
We first remark  that  up to homotopy we have the equality 
$$
\ell=r\circ \pi .$$
This is a very classic result; since we could not find a quotable reference we are going to briefly
discuss its proof. We need to show that the pull-back of the universal bundle
 $E\Gamma\to B\Gamma$ by the two maps
  $r\circ \pi$ and $ \ell$ are isomorphic principal $\Gamma$-bundles over $M$.\\
Let $\pi^*\widetilde{V}$ be the pull back of $b:\widetilde{V}\rightarrow V$. First we point out  that $\pi^*\widetilde{V}$ is path-connected. Moreover, as $\pi_*:\pi_1(M)\rightarrow \pi_1(V)$ is an isomorphism, we have that   $\pi^*\widetilde{V}$ is a simply connected Galois covering of $M$; this means that it is, up to isomorphism, the universal covering of $M$. This latter property  is well known but we give a justification nevertheless. Let $\psi:\pi^*\widetilde{V}\rightarrow \widetilde{V}$ be the map defined by $\psi(x,z)=z$. Then $b\circ \psi=\pi\circ e$ where $e:\pi^*\widetilde{V}\rightarrow M$ is the covering map.  Let $x\in M$ and $y\in \pi^*\widetilde{V}$ with $e(y)=x$ be arbitrarily fixed. Consider $\pi_1(\pi^*\widetilde{V},y)$ and let $[\gamma]\in \pi_1(\pi^*\widetilde{V},y)$.  Then $b_*(\psi_*([\gamma]))=[0]$ as $\widetilde{V}$ is simply connected. Therefore $\pi_*(e_*([\gamma]))=[0]$. But this allows us to conclude that $[\gamma]=0$ as $\pi_*:\pi_1(M,x)\rightarrow \pi_1(V,\pi(x))$ is an isomorphism and $e_*:\pi_1(\pi^*\widetilde{V})\rightarrow \pi_1(M,x)$ is injective. So we showed that $\pi_1(\pi^*\widetilde{V},y)$ is trivial and thus, since $\pi^*\widetilde{V}$ is path-connected, we can conclude that $\pi^*\widetilde{V}$ is simply connected. Summarizing, we can deduce the existence of an isomorphism of coverings $\xi:\widetilde{M}\rightarrow \pi^*\widetilde{V}$. Moreover $\xi:\widetilde{M}\rightarrow \pi^*\widetilde{V}$ is equivariant with respect to the right action of $\Gamma$, that is, the monodromy action of the fundamental group, see for instance \cite{ManettiTop}. Hence $\xi:\widetilde{M}\rightarrow \pi^*\widetilde{V}$ is an isomorphism of $\Gamma$-principal bundles and  we can therefore 
conclude 
%
%
%
that $\pi \circ r=\ell$ up to homotopy.
We have proved in Th. \ref{Kequalities} that
$\pi_*[\overline{\eth}_M]=[\overline{\eth}_{\mathrm{rel}}]$. Thus  $r_* (\pi_* [\overline{\eth}_M])=r_* [\overline{\eth}^V_{\mathrm{rel}}]$ and since  in K-homology 
 $r_* \circ \pi_*= \ell_*$ we conclude that 
$\ell_*[\overline{\eth}_{M}]=r_*[\overline{\eth}^V_{\mathrm{rel}}]$ as required.

\end{proof}

\begin{proof}(of Prop. \ref{sTability}).  
We have already observed that $M$ and $N$ are birationally equivalent
through a birational map $\lambda: N\dashrightarrow M$ which is the composition of $\pi$, $\psi$ and a birational
inverse of $\rho$.  Let $\ell: M\to B\Gamma$ and $\kappa: N\to B\Gamma$ be the classifying maps
of the universal coverings of $M$ and $N$.
We then have, by the Lemma and by \eqref{BLWE},
$$s_*[\overline{\eth}^W_{\mathrm{rel}}]=
\kappa_* [\overline{\eth}_{N}]= 
\ell_* [\overline{\eth}_{M}]=
r_*[\overline{\eth}^V_{\mathrm{rel}}]$$
which is what we wanted to show.

\end{proof}

\begin{corollary}
Let $V$ and $W$ be as in Prop. \ref{sTability}.
Assume in addition that $V$ and $W$ have only rational singularities. Then 
 $$\{\langle \alpha, r_* {\rm Td}^{{\rm BFM}}_* (V) \rangle\,,\quad \alpha\in H^* (B\Gamma,\mathbb{Q})\}$$
 are birational invariants.
 \end{corollary}

 \begin{proof}
 We can either proceed analytically or topologically. In the first case we use Proposition \ref{prop-equality-rational},
 and more precisely \eqref{eq:equality-rational}, and Proposition \ref{sTability} in order to see that 
 $r_* {\rm Td}^{{\rm BFM}}_* (V)=s_* {\rm Td}^{{\rm BFM}}_* (W)$ in $H^* (B\Gamma,\mathbb{Q})$. Consequently, for any $\alpha\in H^* (B\pi_1(W),\mathbb{Q})$ we have $$\langle \alpha, r_* {\rm Td}^{{\rm BFM}}_* (V) \rangle= \langle \alpha, s_* {\rm Td}^{{\rm BFM}}_* (W) \rangle$$ as required.\\ We can also proceed without using analytic classes; indeed, from \eqref{chain2-bis}
 we know that $\pi_*  {\rm Td}_* (M)= {\rm Td}^{{\rm BFM}}_* (V)$ and similarly  $\rho_*  {\rm Td}_* (N)= {\rm Td}^{{\rm BFM}}_* (W)$. Let $z: N\to B\Gamma$ be a classifying map for
 the universal covering of $N$. As $\ell$ is homotopic to $r\circ \pi$ and $z$ is homotopic to $s\circ \rho$ we infer 
 from  \cite[Proposition 1.4]{WBlock} that $r_* {\rm Td}^{{\rm BFM}}_* (V)=s_* {\rm Td}^{{\rm BFM}}_* (W)$ in $H^* (B\Gamma,\mathbb{Q})$. Consequently we have again that for any $\alpha\in H^* (B\Gamma,\mathbb{Q})$  the equality $\langle \alpha, r_* {\rm Td}^{{\rm BFM}}_* (V) \rangle= \langle \alpha, s_* {\rm Td}^{{\rm BFM}}_* (W) \rangle$ holds. 
\end{proof}

As mentioned in the introduction examples of singular projective varieties admitting a resolution that induces an isomorphism between fundamental groups are for instance projective surfaces with only rational singularities and  projective varieties with log-terminal singularities, see for instance  \cite{Briesk} and  \cite{Takayama}, respectively. Other examples are provided by  complex projective varieties with quotient singularities, see \cite{Kollar}.

\bibliographystyle{plain}
\bibliography{atodd}

\begin{thebibliography}{10}

\bibitem{AlbinGell}
Pierre Albin and Jesse Gell-Redman.
\newblock The index formula for families of dirac type operators on
  pseudomanifolds.
\newblock https://arxiv.org/abs/1712.08513.

\bibitem{ALMP11}
Pierre Albin, Eric Leichtnam, Rafe Mazzeo, and Paolo Piazza.
\newblock The signature package on {W}itt spaces.
\newblock {\em Ann. Sci. \'Ec. Norm. Sup\'er. (4)}, 45(2):241--310, 2012.

\bibitem{ALMP13.2}
Pierre Albin, Eric Leichtnam, Rafe Mazzeo, and Paolo Piazza.
\newblock The {N}ovikov conjecture on {C}heeger spaces.
\newblock {\em J. Noncommut. Geom.}, 11(2):451--506, 2017.

\bibitem{ALMP13.1}
Pierre Albin, Eric Leichtnam, Rafe Mazzeo, and Paolo Piazza.
\newblock Hodge theory on {C}heeger spaces.
\newblock {\em J. Reine Angew. Math.}, 744:29--102, 2018.

\bibitem{Artin}
Michael Artin.
\newblock On isolated rational singularities of surfaces.
\newblock {\em Amer. J. Math.}, 88:129--136, 1966.

\bibitem{Baaj-Julg}
Saad Baaj and Pierre Julg.
\newblock Th{\'e}orie bivariante de {K}asparov et op{\'e}rateurs non born{\'e}s
  dans les {$C^*$}-modules hilbertiens.
\newblock {\em C. R. Math. Acad. Sci. Paris}, 296(21):875--878, 1983.

\bibitem{compactcomplexsurfaces}
Wolf~P. Barth, Klaus Hulek, Chris A.~M. Peters, and Antonius Van~de Ven.
\newblock {\em Compact complex surfaces}, volume~4 of {\em Ergebnisse der
  Mathematik und ihrer Grenzgebiete. 3. Folge. A Series of Modern Surveys in
  Mathematics [Results in Mathematics and Related Areas. 3rd Series. A Series
  of Modern Surveys in Mathematics]}.
\newblock Springer-Verlag, Berlin, second edition, 2004.

\bibitem{BFMI}
Paul Baum, William Fulton, and Robert MacPherson.
\newblock Riemann-{R}och for singular varieties.
\newblock {\em Publ. Math. Inst. Hautes \'Etudes Sci.}, (45):101--145, 1975.

\bibitem{BFMII}
Paul Baum, William Fulton, and Robert MacPherson.
\newblock Riemann-{R}och and topological {$K$}\ theory for singular varieties.
\newblock {\em Acta Math.}, 143(3-4):155--192, 1979.

\bibitem{BaHiSc}
Paul Baum, Nigel Higson, and Thomas Schick.
\newblock On the equivalence of geometric and analytic {$K$}-homology.
\newblock {\em Pure Appl. Math. Q.}, 3(1, Special Issue: In honor of Robert D.
  MacPherson. Part 3):1--24, 2007.

\bibitem{Bei2017}
Francesco Bei.
\newblock On the {L}aplace-{B}eltrami operator on compact complex spaces.
\newblock arxiv.org/abs/1706.05962. To appear on {\em Trans. Amer. Math. Soc.}

\bibitem{FBJTA}
Francesco Bei.
\newblock On the {$L^2$}-{P}oincar\'e duality for incomplete {R}iemannian
  manifolds: a general construction with applications.
\newblock {\em J. Topol. Anal.}, 8(1):151--186, 2016.

\bibitem{FraBei}
Francesco Bei.
\newblock Degenerating {H}ermitian metrics and spectral geometry of the
  canonical bundle.
\newblock {\em Adv. Math.}, 328:750--800, 2018.

\bibitem{SympBei}
Francesco Bei.
\newblock Symplectic manifolds, {$L^p$}-cohomology and {$q$}-parabolicity.
\newblock {\em Differential Geom. Appl.}, 64:136--157, 2019.

\bibitem{BeiPiazza}
Francesco Bei and Paolo Piazza.
\newblock On the {$L^2$}-{$\overline{\partial}$}-cohomology of certain complete
  {K}\"{a}hler metrics.
\newblock {\em Math. Z.}, 290(1-2):521--537, 2018.

\bibitem{BieMil}
Edward Bierstone and Pierre~D. Milman.
\newblock Canonical desingularization in characteristic zero by blowing up the
  maximum strata of a local invariant.
\newblock {\em Invent. Math.}, 128(2):207--302, 1997.

\bibitem{WBlock}
Jonathan Block and Shmuel Weinberger.
\newblock Higher {T}odd classes and holomorphic group actions.
\newblock {\em Pure Appl. Math. Q.}, 2(4, Special Issue: In honor of Robert D.
  MacPherson. Part 2):1237--1253, 2006.

\bibitem{BraSchYou}
Jean-Paul Brasselet, J\"{o}rg Sch\"{u}rmann, and Shoji Yokura.
\newblock Hirzebruch classes and motivic {C}hern classes for singular spaces.
\newblock {\em J. Topol. Anal.}, 2(1):1--55, 2010.

\bibitem{Briesk}
Egbert Brieskorn.
\newblock Rationale {S}ingularit\"aten komplexer {F}l\"achen.
\newblock {\em Invent. Math.}, 4:336--358, 1967/1968.

\bibitem{BruLe}
J.~Br{\"u}ning and M.~Lesch.
\newblock Hilbert complexes.
\newblock {\em J. Funct. Anal.}, 108(1):88--132, 1992.

\bibitem{BPScurves}
Jochen Br\"uning, Norbert Peyerimhoff, and Herbert Schr\"oder.
\newblock The {$\overline\partial$}-operator on algebraic curves.
\newblock {\em Comm. Math. Phys.}, 129(3):525--534, 1990.

\bibitem{Cheeger:Spec}
Jeff Cheeger.
\newblock Spectral geometry of singular {R}iemannian spaces.
\newblock {\em J. Differential Geom.}, 18(4):575--657 (1984), 1983.

\bibitem{Cheltsov}
I.~A. Chelt'sov.
\newblock Del {P}ezzo surfaces with nonrational singularities.
\newblock {\em Mat. Zametki}, 62(3):451--467, 1997.

\bibitem{Esnault}
H\'el\`ene Esnault and Eckart Viehweg.
\newblock Logarithmic de {R}ham complexes and vanishing theorems.
\newblock {\em Invent. Math.}, 86(1):161--194, 1986.

\bibitem{Fischer}
Gerd Fischer.
\newblock {\em Complex analytic geometry}.
\newblock Lecture Notes in Mathematics, Vol. 538. Springer-Verlag, Berlin-New
  York, 1976.

\bibitem{FoxHaskellHodge}
Jeffrey Fox and Peter Haskell.
\newblock Hodge decompositions and {D}olbeault complexes on normal surfaces.
\newblock {\em Trans. Amer. Math. Soc.}, 343(2):765--778, 1994.

\bibitem{FoxHaskellTodd}
Jeffrey Fox and Peter Haskell.
\newblock Perturbed {D}olbeault operators and the homology {T}odd class.
\newblock {\em Proc. Amer. Math. Soc.}, 128(12):3715--3721, 2000.

\bibitem{FultonLibro}
William Fulton.
\newblock {\em Intersection theory}, volume~2 of {\em Ergebnisse der Mathematik
  und ihrer Grenzgebiete. 3. Folge. A Series of Modern Surveys in Mathematics
  [Results in Mathematics and Related Areas. 3rd Series. A Series of Modern
  Surveys in Mathematics]}.
\newblock Springer-Verlag, Berlin, second edition, 1998.

\bibitem{MellesMilmanPacific}
Caroline Grant and Pierre Milman.
\newblock Metrics for singular analytic spaces.
\newblock {\em Pacific J. Math.}, 168(1):61--156, 1995.

\bibitem{MellesMilmanToulouse}
Caroline Grant~Melles and Pierre Milman.
\newblock Classical {P}oincar\'e metric pulled back off singularities using a
  {C}how-type theorem and desingularization.
\newblock {\em Ann. Fac. Sci. Toulouse Math. (6)}, 15(4):689--771, 2006.

\bibitem{GraRe}
Hans Grauert and Reinhold Remmert.
\newblock {\em Coherent analytic sheaves}, volume 265 of {\em Grundlehren der
  Mathematischen Wissenschaften [Fundamental Principles of Mathematical
  Sciences]}.
\newblock Springer-Verlag, Berlin, 1984.

\bibitem{SingDef}
G.-M. Greuel, C.~Lossen, and E.~Shustin.
\newblock {\em Introduction to singularities and deformations}.
\newblock Springer Monographs in Mathematics. Springer, Berlin, 2007.

\bibitem{GrieserLesch}
Daniel Grieser and Matthias Lesch.
\newblock On the {$L^2$}-{S}tokes theorem and {H}odge theory for singular
  algebraic varieties.
\newblock {\em Math. Nachr.}, 246/247:68--82, 2002.

\bibitem{Griffiths-Harris}
Phillip Griffiths and Joseph Harris.
\newblock {\em Principles of algebraic geometry}.
\newblock Wiley Classics Library. John Wiley \& Sons, Inc., New York, 1994.
\newblock Reprint of the 1978 original.

\bibitem{HaskellIndex}
Peter Haskell.
\newblock Index theory on curves.
\newblock {\em Trans. Amer. Math. Soc.}, 288(2):591--604, 1985.

\bibitem{Haskell-K}
Peter Haskell.
\newblock Index theory of geometric {F}redholm operators on varieties with
  isolated singularities.
\newblock {\em $K$-Theory}, 1(5):457--466, 1987.

\bibitem{Hilsumproj}
Michel Hilsum.
\newblock Une preuve analytique de la conjecture de {J}.{ R}osenberg.
\newblock https://hal.archives-ouvertes.fr/hal-01841905v1.

\bibitem{Hilsum-LNM}
Michel Hilsum.
\newblock Signature operator on {L}ipschitz manifolds and unbounded {K}asparov
  bimodules.
\newblock In {\em Operator algebras and their connections with topology and
  ergodic theory ({B}uteni, 1983)}, volume 1132 of {\em Lecture Notes in
  Math.}, pages 254--288. Springer, Berlin, 1985.

\bibitem{Hiro}
Heisuke Hironaka.
\newblock Resolution of singularities of an algebraic variety over a field of
  characteristic zero. {I}, {II}.
\newblock {\em Ann. of Math. (2) {\bf 79} (1964), 109--203; ibid. (2)},
  79:205--326, 1964.

\bibitem{Kasparov}
G.~G. Kasparov.
\newblock Topological invariants of elliptic operators. {I}. {$K$}-homology.
\newblock {\em Izv. Akad. Nauk SSSR Ser. Mat.}, 39(4):796--838, 1975.

\bibitem{Kasparov-inventiones}
Gennadi Kasparov.
\newblock Equivariant {KK}-theory and the {N}ovikov conjecture.
\newblock {\em Invent. Math.}, 91:147--201, 1988.

\bibitem{Kollar}
J\'{a}nos Koll\'{a}r.
\newblock Shafarevich maps and plurigenera of algebraic varieties.
\newblock {\em Invent. Math.}, 113(1):177--215, 1993.

\bibitem{JKol}
J\'anos Koll\'ar.
\newblock {\em Singularities of the minimal model program}, volume 200 of {\em
  Cambridge Tracts in Mathematics}.
\newblock Cambridge University Press, Cambridge, 2013.
\newblock With a collaboration of S\'andor Kov\'acs.

\bibitem{KoMo}
J\'anos Koll\'ar and Shigefumi Mori.
\newblock {\em Birational geometry of algebraic varieties}, volume 134 of {\em
  Cambridge Tracts in Mathematics}.
\newblock Cambridge University Press, Cambridge, 1998.
\newblock With the collaboration of C. H. Clemens and A. Corti, Translated from
  the 1998 Japanese original.

\bibitem{Laufer}
Henry~B. Laufer.
\newblock On rational singularities.
\newblock {\em Amer. J. Math.}, 94:597--608, 1972.

\bibitem{LawMich}
H.~Blaine Lawson, Jr. and Marie-Louise Michelsohn.
\newblock {\em Spin geometry}, volume~38 of {\em Princeton Mathematical
  Series}.
\newblock Princeton University Press, Princeton, NJ, 1989.

\bibitem{Roni}
Roni~N. Levy.
\newblock The {R}iemann-{R}och theorem for complex spaces.
\newblock {\em Acta Math.}, 158(3-4):149--188, 1987.

\bibitem{LiTian}
Peter Li and Gang Tian.
\newblock On the heat kernel of the {B}ergmann metric on algebraic varieties.
\newblock {\em J. Amer. Math. Soc.}, 8(4):857--877, 1995.

\bibitem{Lott}
John Lott.
\newblock A {D}olbeault-{H}ilbert complex for a variety with isolated singular
  points.
\newblock arxiv.org/abs/1904.07744. To appear on {\em Ann. K-Theory}.

\bibitem{ManettiTop}
Marco Manetti.
\newblock {\em Topology}, volume~91 of {\em Unitext}.
\newblock Springer, Cham, 2015.
\newblock Translated from the 2014 Italian edition by Simon G. Chiossi, La
  Matematica per il 3+2.

\bibitem{Nagase}
Masayoshi Nagase.
\newblock Remarks on the {$L^2$}-{D}olbeault cohomology groups of singular
  algebraic surfaces and curves.
\newblock {\em Publ. Res. Inst. Math. Sci.}, 26(5):867--883, 1990.

\bibitem{Takeo}
Takeo Ohsawa.
\newblock {\em {$L^2$} approaches in several complex variables}.
\newblock Springer Monographs in Mathematics. Springer, Tokyo, 2015.
\newblock Development of Oka-Cartan theory by $L^2$ estimates for the
  $\overline{\partial}$ operator.

\bibitem{OvRu}
N.~{\O}vrelid and J.~Ruppenthal.
\newblock {$L^2$}-properties of the {$\overline{\partial}$} and the
  {$\overline{\partial}$}-{N}eumann operator on spaces with isolated
  singularities.
\newblock {\em Math. Ann.}, 359(3-4):803--838, 2014.

\bibitem{OvreVaAJM}
Nils {\O}vrelid and Sophia Vassiliadou.
\newblock Some {$L^2$} results for {$\overline\partial$} on projective
  varieties with general singularities.
\newblock {\em Amer. J. Math.}, 131(1):129--151, 2009.

\bibitem{OvreVaIn}
Nils {\O}vrelid and Sophia Vassiliadou.
\newblock {$L^2$}-{$\overline\partial$}-cohomology groups of some singular
  complex spaces.
\newblock {\em Invent. Math.}, 192(2):413--458, 2013.

\bibitem{PSCrelle}
William Pardon and Mark Stern.
\newblock Pure {H}odge structure on the {$L_2$}-cohomology of varieties with
  isolated singularities.
\newblock {\em J. Reine Angew. Math.}, 533:55--80, 2001.

\bibitem{Pardon}
William~L. Pardon.
\newblock The {$L_2$}-{$\overline\partial$}-cohomology of an algebraic surface.
\newblock {\em Topology}, 28(2):171--195, 1989.

\bibitem{PardonSternJAMS}
William~L. Pardon and Mark~A. Stern.
\newblock {$L^2\text{--}\overline\partial$}-cohomology of complex projective
  varieties.
\newblock {\em J. Amer. Math. Soc.}, 4(3):603--621, 1991.

\bibitem{PiZe}
Paolo Piazza and Vito~Felice Zenobi.
\newblock Singular spaces, groupoids and metrics of positive scalar curvature.
\newblock {\em J. Geom. Phys.}, 137:87--123, 2019.

\bibitem{Miles}
Miles Reid.
\newblock Canonical {$3$}-folds.
\newblock In {\em Journ\'ees de {G}\'eometrie {A}lg\'ebrique d'{A}ngers,
  {J}uillet 1979/{A}lgebraic {G}eometry, {A}ngers, 1979}, pages 273--310.
  Sijthoff \&\ Noordhoff, Alphen aan den Rijn---Germantown, Md., 1980.

\bibitem{RosenbergPSC}
Jonathan Rosenberg.
\newblock {$C^{\ast} $}-algebras, positive scalar curvature, and the {N}ovikov
  conjecture.
\newblock {\em Publ. Math. Inst. Hautes \'Etudes Sci.}, (58):197--212 (1984),
  1983.

\bibitem{RosenbergPSCIII}
Jonathan Rosenberg.
\newblock {$C^\ast$}-algebras, positive scalar curvature, and the {N}ovikov
  conjecture. {III}.
\newblock {\em Topology}, 25(3):319--336, 1986.

\bibitem{rosenberg-tams}
Jonathan Rosenberg.
\newblock An analogue of the {N}ovikov conjecture in complex algebraic
  geometry.
\newblock {\em Trans. Amer. Math. Soc.}, 360(1):383--394, 2008.

\bibitem{RUJFA}
J.~Ruppenthal.
\newblock Compactness of the {$\overline\partial$}-{N}eumann operator on
  singular complex spaces.
\newblock {\em J. Funct. Anal.}, 260(11):3363--3403, 2011.

\bibitem{JRupp}
J.~Ruppenthal.
\newblock {$L^2$}-theory for the {$\overline\partial$}-operator on compact
  complex spaces.
\newblock {\em Duke Math. J.}, 163(15):2887--2934, 2014.

\bibitem{RUIMRN}
J.~Ruppenthal.
\newblock {$L^2$}-{S}erre duality on singular complex spaces and rational
  singularities.
\newblock {\em Int. Math. Res. Not. IMRN}, (23):7198--7240, 2018.

\bibitem{Saperisolated}
Leslie Saper.
\newblock {$L_2$}-cohomology of {K}\"ahler varieties with isolated
  singularities.
\newblock {\em J. Differential Geom.}, 36(1):89--161, 1992.

\bibitem{SiuYaucompact}
Yum~Tong Siu and Shing~Tung Yau.
\newblock Compactification of negatively curved complete {K}\"ahler manifolds
  of finite volume.
\newblock In {\em Seminar on {D}ifferential {G}eometry}, volume 102 of {\em
  Ann. of Math. Stud.}, pages 363--380. Princeton Univ. Press, Princeton, N.J.,
  1982.

\bibitem{Takayama}
Shigeharu Takayama.
\newblock Local simple connectedness of resolutions of log-terminal
  singularities.
\newblock {\em Internat. J. Math.}, 14(8):825--836, 2003.

\bibitem{Zucker}
Steven Zucker.
\newblock Hodge theory with degenerating coefficients. {$L_{2}$}\ cohomology in
  the {P}oincar\'e metric.
\newblock {\em Ann. of Math. (2)}, 109(3):415--476, 1979.

\end{thebibliography}

\end{document}